\documentclass[smallextended]{amsart}
\usepackage[top=1in, bottom=1in, left=1.25in, right=1.25in]{geometry}
\usepackage{mathrsfs}
\usepackage{stackrel,amssymb}
\usepackage{amsfonts}
\usepackage[linkcolor=red,citecolor=blue]{hyperref}
\usepackage{latexsym}
\usepackage[all,cmtip]{xy}
\usepackage{amsmath}
\usepackage{color,ifsym}
\usepackage{graphicx}
\usepackage{fancybox}
\usepackage{longtable}
\usepackage{lscape}
\usepackage{float}
\usepackage{mathdots}
\usepackage{enumerate}
\usepackage{multirow}

\newtheorem{thm}{Theorem}
\newtheorem{lem}[thm]{Lemma}        
\newtheorem{cor}[thm]{Corollary}

\newtheorem{mthm}[thm]{Main Theorem}

\newtheorem{rem}{{\it Remark}} 

\theoremstyle{plain} 
\newcommand{\thistheoremname}{}
\newtheorem*{genericthm*}{\thistheoremname}
\newenvironment{namedthm*}[1]
{\renewcommand{\thistheoremname}{#1}%
	\begin{genericthm*}}
	{\end{genericthm*}}

\begin{document}
	\title[Casselman--Shahidi conjecture]{Casselman--Shahidi's conjecture on normalized intertwining operators for groups of classical type}
   
	\author{CAIHUA LUO}
	\address{department of mathematics \\ bar-ilan university \\ ramat-gan, 5290000\\ israel}
	\email{chluo@amss.ac.cn}
	\date{}
	\dedicatory{\bf Dedicated to Professor Wee Teck Gan on the occasion of his 50th birthday}
\maketitle
\begin{abstract}
	Intertwining operators play an essential role and appear everywhere in the Langlands program, their analytic properties interact directly, yet deeply with the decomposition of parabolic induction locally and the residues of Eisenstein series globally. Inspired by the profound Langlands--Shahidi theory, Casselman--Shahidi conjectured that a certain normalization factor would govern the singularity of intertwining operators for generic standard modules. Indeed, motivated by the theory of theta correspondence, especially the Siegel--Weil formula globally, and the composition problem of degenerate principal series and the demand of a g.c.d. definition of standard $L$-functions in the framework of the doubling method locally, an optimal normalization factor has been determined for degenerate principal series of classical groups via the theory of integrals on prehomogeneous vector spaces. Such a method seems impossible to be generalized to work even in the setting of degenerate generalized principal series, which are naturally involved in the recent Cai--Friedberg--Ginzburg--Kaplan's generalized doubling method. To circumvent it, we discover a new uniform argument that can answer the singularity problem of intertwining operators for a large class of induced representations. As an illustration, we prove the aforementioned Casselman--Shahidi conjecture for quasi-split groups of classical type in the paper. Along the way, with the help of Shahidi's local coefficient theory, we also prove that those normalized intertwining operators are always non-zero, and provide a new one-sentence proof of the standard module conjecture in the spirit of Casselman--Shahidi.
\end{abstract}

\section{introduction}
For simplicity, we first let $G_n$ be a quasi-split classical group defined over a $p$-adic field $F$, and $P_k=M_kN_k$ be its standard maximal parabolic subgroup with Levi subgroup $M_k\simeq GL_k\times G_{n_0}$. The setup for general connected reductive quasi-split groups will be discussed in Section \ref{sectiongeneral}. For $\tau_a$ (resp. $\sigma_\gamma$) a discrete series representation of $GL_k$ (resp. $G_{n_0}$), assume $\sigma_\gamma$ is also generic, Casselman--Shahidi conjectured that the following normalized standard intertwining operator
\[M^*(s,\tau_a,\sigma_\gamma):=\frac{1}{\alpha(s,\tau_a,\sigma_\gamma)}M(s,\tau_a,\sigma_\gamma):~\tau_a|det(\cdot)|^s\rtimes\sigma_\gamma\longrightarrow\tau_a^*|det(\cdot)|^{s^*}\rtimes\sigma_\gamma \]   
is holomorphic in $s\in\mathbb{C}$. Here $*$ means (conjugate) contragredient and $s^*$ equals (conjugate) -s. Please refer to \cite[P.563 (3)]{casselman1998irreducibility} or the next section for the details. On the other hand, analogous holomorphicity problems appear often in the development of Langlands program. Herein we only mention three aspects in which our argument could potentially make a direct contribution in our future work as follows.
\begin{enumerate}[(i).]
	\item Theta correspondence theory: Harris, Kudla, Rallis, Sweet, and Yamana's works on the composition problem of degenerate principal series for classical groups and the characterization of non-vanishing of theta-lifting \cite{kudlarallis1992,kudlasweet1997,harriskudlasweet1996,yamana2011,yamanatheta}. With the help of \cite{muicsavin2000quaternion}, a direct application of our argument would be on the composition problem of degenerate generalized principal series for classical groups, especially the under-investigated quaternion unitary groups.
	\item Residual spectrum: the reducibility of parabolic induction locally and the location of poles of Eisenstein series globally for classical groups in the framework of Jacquet and M{\oe}glin--Waldspurger \cite{jacquet1984residualGL,moeglin1989residue}. Such a basic problem has been studied extensively especially by Arthur, M{\oe}glin, and Kim among others (for example see \cite{arthur1984problemtraceformula,arthur1989unipotentautomorphicconjecture,arthur1990unipotentmotivation,moeglin1991ICM,moeglin1991discretespectrumUnipotent,moeglin2001residualconjecture,moeglin2008automorphicsquareintegrable,moeglin2011intertwiningeisenstein,kim2001residualOrthogonal,kim2017residualU(2n),grbac2009residual,jiangliuzhang2013,miller2013residual,hanzer2018nonsiegeleisenstein}). One should note that Arthur's ingenious approach, in terms of A-packets, acts as a guiding principle but lacks of concreteness, this has been complemented by a series of works of M{\oe}glin, M{\oe}glin--Renard, and Bin Xu among others (see \cite{moeglin2006certainpackets,moeglin2006paquets,moeglin2009arthurpackets,moeglin2011multiplicityone,moeglinrenard2018arthurpackets,moeglinrenard2018arthurpacketsReal,moeglinrenard2020arthurpacketsReal,xu2017moeglinparametrization,xu2021combinatorialArthurpackets,atobe2020construction}). Given the fruitful results obtained for split classical groups, it is curious to see what we could say for quaternion unitary groups. It seems possible to first generalize at least Part I of \cite{moeglin1989residue} to the setting of general linear groups over division algebras (see Remark \ref{moeglin}).
	\item The theory of integral representations of $L$-functions: a g.c.d. definition of local tensor product $L$-functions in the frameworks of the Rankin--Selberg method and the recent Cai--Friedberg--Ginzburg--Kaplan's generalized doubling method \cite{kaplan2013gcd,cai2018doubling,cai2019doubling}. Such a pursuit is motivated by the need to obtain some information of the possible poles of global $L$-functions which play an important role in the above (i) and (ii). First, our Main Theorem \ref{mainthm} removes the assumption in \cite[Theorem 1.1]{kaplan2013gcd}. Secondly, I have worked out the part of the g.c.d. definition of tensor product $L$-functions for tempered representations under the framework of the generalized doubling method for classical groups (see \cite{luo2021gcd}). On the other hand, great attention has been attracted to the ambitious Braverman--Kazhdan/Ng{\^o} program. This is a mega-project to systematically construct automporphic $L$-functions in the sense of Langlands' functoriality conjecture, which grows from generalizing Godement--Jacquet's zeta integral for standard $L$-functions of general linear groups (see \cite{bravermankazhdan2000gammafunctions,bravermankazhdan2002intertwining,sakellaridis,lafforgue2014Lfunctions,ngo2014Lfunction}). Recently, works have been done to rewrite the well-developed Rankin--Selberg method and doubling method in the language of this program (see \cite{li2018zetaintegrals,shahidi2018Lfunctions,getzliu2021refined,jiangluozhang2020Symplectic}). In view of Getz--Liu's approach via Mellin inversion of excellent sections in the doubling method, it may be possible to extend their result to the setting of the recent generalized doubling method. However, a more interesting way to put the generalized doubling method into the Braverman--Kazhdan/Ng{\^o} program might be to follow Jiang--Luo--Zhang's work.
\end{enumerate} 
Specifically, we want to point out that the original path, adopted in all previous works, to attack the holomorphicity problem for degenerate principal series is to analyze integrals on prehomogeneous vector spaces (see \cite{piatetskishapirotriple,kudlarallis1992,ikeda1992location,sweet1995,yamana2011}). Such a method seems hard to be generalized to the setting of degenerate generalized principal series, whilst our argument illustrated in the paper works almost perfectly in this setting and beyond. On the other hand, one may note that attentions have since been focused much more on the analogous problem for Langlands' conjectural fully normalized intertwining operators, which are designed to preserve the unitarity and the multiplicative relations with respect to the decompositions of Weyl elements as proved by Shahidi for generic representations in \cite{shahidi1990proof}, but analyzing such a normalization seems not good enough to detect the possible poles of standard intertwining operators on the whole complex plane (see \cite{zhang1997localintertwining,jantzenkim2001intertwiningimage,cogdellkimshapiroshahidi2001,cogdellkimshapiroshahidi2004,kim2005intertwiningLfunction,moeglin2010holomorphy,moeglin2011intertwiningeisenstein,kim2011local,heiermann2013tempered}). However, as discussed in (i,ii,iii), it is pivotal and would be beneficial to study directly, if we could, the holomorphicity problem of the ad hoc ``half-normalized'' intertwining operators in the sense of Casselman--Shahidi. Indeed, it turns out that we do have a way, inspired mainly by the arguments in \cite[Part I]{moeglin1989residue}, to handle it at least in the setting of classical groups, which is based on a simple observation of a ``non-symmetry'' property of the normalization factors appearing in the associated reduced decompositions of Weyl elements with respect to different embeddings of inducing data. By doing so, this gives us a sense that we may have discovered a uniform way to attack the holomorphicity problem of ``half-normalized'' intertwining operators for arbitrary induced representations if normalization factors are given, or saying it in another way, a uniform way to find the poles of standard intertwining operators via reverse engineering our argument in the paper. One may also see that those normalized intertwining operators discussed in the paper are non-zero, which follows from Shahidi's local coefficient theory (see Corollary \ref{Cornonzero}). But it may not be an easy task to see that the non-zero property of Casselman--Shahidi's normalized intertwining operators holds in general (see \cite{luo2021holomorphicityClassicalGroup,luo2021reducibilitypointsGL}). We hope one could get some information for this question with the help of the tool of Jacquet module or analogues local coefficient theory (see \cite{friedberggoldberg1999nongeneric,cai2018doubling}).

The novel part of our argument is an observation of an intrinsic non-symmetry property of normalization factors involved in different reduced decompositions of the intertwining operator $M(s,\tau_a,\sigma_\gamma)$ corresponding to different embeddings of our inducing data $\tau_a$ and $\sigma_{\gamma}$. To demonstrate the elementary idea, we provide an example of the non-symmetry property as follows.

\underline{(Way 1: $\tau_a\hookrightarrow|det(\cdot)|^{\frac{a-1}{2}}\tau\times |det(\cdot)|^{-\frac{1}{2}}\tau_{a-1})$}
\[\xymatrix{ |det(\cdot)|^{s}\tau_a\rtimes \sigma_\gamma \ar@{^(->}[r] \ar[ddd]_{M(s,\tau_a,\sigma_\gamma)} & |det(\cdot)|^{s+\frac{a-1}{2}}\tau\times\mbox{$\underbrace{|det(\cdot)|^{s-\frac{1}{2}}\tau_{a-1}\rtimes \sigma_\gamma}$}\ar[d]^{M(s-\frac{1}{2},\tau_{a-1},\sigma_\gamma)}\\
	&\mbox{$\underbrace{|det(\cdot)|^{s+\frac{a-1}{2}}\tau\times|det(\cdot)|^{-s+\frac{1}{2}}\tau_{a-1}}$}\rtimes \sigma_\gamma  \ar[d]^{M_{GL}(\cdots)}\\
	&|det(\cdot)|^{-s+\frac{1}{2}}\tau_{a-1}\times\mbox{$\underbrace{|det(\cdot)|^{s+\frac{a-1}{2}}\tau\rtimes\sigma_\gamma}$} \ar[d]^{M(s+\frac{a-1}{2},\tau,\sigma_\gamma)}\\ 
	|det(\cdot)|^{-s}\tau_a\rtimes \sigma_\gamma \ar@{^(->}[r] & |det(\cdot)|^{-s+\frac{1}{2}}\tau_{a-1}\times|det(\cdot)|^{-s-\frac{a-1}{2}}\tau\rtimes\sigma_\gamma.
}\] 

\underline{(Way 3: $\sigma_\gamma\hookrightarrow|det(\cdot)|^{\frac{r-1}{4}}\tau\rtimes \sigma_{\gamma'})$}
\[\xymatrix{|det(\cdot)|^{s}\tau_a\rtimes \sigma_{\gamma} \ar@{^(->}[r] \ar[ddd]_{M(s,\tau_a,\sigma_{\gamma})} & \mbox{$\underbrace{|det(\cdot)|^{s}\tau_{a}\times|det(\cdot)|^{\frac{r-1}{4}}\tau}$}\rtimes \sigma_{\gamma'}\ar[d]^{M_{GL}(\cdots)}\\
	&|det(\cdot)|^{\frac{r_1-1}{4}}\tau\times\mbox{$\underbrace{|det(\cdot)|^{s}\tau_{a}\rtimes\sigma_{\gamma'}}$} \ar[d]^{M(s,\tau_a,\sigma_{\gamma'})}\\ 
	&\mbox{$\underbrace{|det(\cdot)|^{\frac{r_1-1}{4}}\tau\times|det(\cdot)|^{-s}\tau_a}$}\rtimes \sigma_{\gamma'}  \ar[d]^{M_{GL}(\cdots)}\\
	|det(\cdot)|^{-s}\tau_a\rtimes \sigma_{\gamma} \ar@{^(->}[r] & |det(\cdot)|^{-s}\tau_a\times|det(\cdot)|^{\frac{r_1-1}{4}}\tau\rtimes\sigma_{\gamma'}.
}\]
Those two commutative diagrams give rise to two discrepancies of normalization factors on the right hand side comparing to the left hand side as follows.
\begin{align*}
P_1:&=\frac{\alpha(s-\frac{1}{2},\tau_{a-1},\sigma_\gamma)\alpha_{GL}(\cdots)\alpha(s+\frac{a-1}{2},\tau,\sigma_\gamma)}{\alpha(s,\tau_a,\sigma_\gamma)}\\
P_3:&=\frac{\alpha_{GL}(\cdots)\alpha(s,\tau_a,\sigma_{\gamma'})\alpha_{GL}(\cdots)}{\alpha(s,\tau_a,\sigma_\gamma)}.
\end{align*}
Here $\alpha_{GL}(\cdot)$ is the normalization factor for $GL$ (see Lemma \ref{lemGL}). Now our elementary idea is 
\[\mbox{\bf if $P_1$ and $P_3$ are co-prime, then $M^*(s,\tau_a,\sigma_\gamma)$ is holomorphic in $s$.}  \]
This results from an induction argument. If not, we first sort out the bad cases, and then carry out a detailed case-by-case analysis by hand. Luckily, the bad cases are pretty easy to handle. As an illustration, we will first prove Cassleman--Shahidi's holomorphicity conjecture for split symplectic groups and special orthogonal groups, i.e., $G_n=Sp_{2n}$, $SO_{2n+1}$ or $SO_{2n}$, via a concrete calculation. Indeed, our setup and argument in the main body of the paper also work for other quasi-split classical groups. Please refer to Remark \ref{rmkforquasi-split} for the necessary modifications. Given this, we can reduce the case of groups of classical type to the special case of classical groups, thus finishing the proof of Casselman--Shahidi's conjecture for all groups of classical type, i.e., Type $A_n$, $B_n$, $C_n$, or $D_n$, in the paper. It seems that the reduced decomposition of Weyl element as we expected is a little bit complicated in the setting of exceptional groups, we leave the case of groups of exceptional type for a future work. At last, we also want to point out that after typing up the manuscript, we realized that some of the arguments could be simplified significantly based on some basic properties of standard intertwining operators for standard modules, but we still want to keep the elementary argument in its full form serving as a template for future work, for the reason that our argument is motivated by, and can be applied to, the holomorphicity problem of a large class of induced representations where those special properties do not hold (see Remark \ref{simplifyargument}).  

The structure of the paper is as follows. In the next section, we prepare some notation and state our Main Theorem \ref{mainthm}, as well as Corollaries \ref{Cornonzero} and \ref{standardmoduleconj}, while its proof will be given in the section next to it. In the last section, we prove Casselman--Shahidi's holomorphicity conjecture in the setting of groups of classical type, i.e., Theorem \ref{mthgeneral}, which is achieved, via an abstract argument, by reducing the problem to the special case of classical groups that we proved earlier.

\section{main theorem}
Let $F$ be a non-archimedean local field of characteristic zero. Denote by $\mathcal{O}$ its
ring of integers and by $\mathcal{P}$ its maximal ideal generated by the uniformizer $\mathfrak{w}$. Let $q$ be the number of elements in the residue field $\mathcal{O}/\mathcal{P}$. We set $|\cdot|$ to be the absolute value of $F$ such that $|\mathfrak{w}|=q^{-1}$. Let $E/F$ be a quadratic field extension of $F$, we use the same terminologies of $F$ for $E$ if there is no confusion. For ease of statement, we first assume $G_n$ is a split classical group of type $B_n$, $C_n$ or $D_n$ of rank $n$ defined over $F$, i.e., $G_n=Sp_{2n}(F)$, $SO_{2n+1}(F)$ or $SO_{2n}(F)$. The quasi-split case will be given in Remark \ref{rmkforquasi-split}.
 
\subsection{Certain induced representations}Let $k$ be a positive integer and $P_k=M_kN_k\subset G_n$ be a standard maximal parabolic subgroup of $G_n$ with its Levi subgroup $M_k\simeq GL_k\times G_{n_0}$, here $n_0=n-k\in \mathbb{N}$. Let $a$ be a positive integer and $\tau$ be a fixed unitary supercuspidal representation of $GL_{\frac{k}{a}}$ with $\frac{k}{a}\in \mathbb{N}$, and denote by $\tau_a$ the associated discrete series of $GL_k$, i.e., the unique subrepresentation of the normalized induced representation
\[ |det(\cdot)|^{\frac{a-1}{2}}\tau\times|det(\cdot)|^{\frac{a-3}{2}}\tau\times\cdots\times |det(\cdot)|^{-\frac{a-1}{2}}\tau:=Ind^{GL_k}(|det(\cdot)|^{\frac{a-1}{2}}\tau\otimes \cdots \otimes|det(\cdot)|^{-\frac{a-1}{2}}\tau).\]
Let $n_{00}$ be a non-negative integer and $\sigma$ be a fixed generic supercuspidal representation of $G_{n_{00}}$ (please refer to the Subsection \ref{genericnotion} for the notion of generic), we say a discrete series $\sigma_\gamma$ of $G_{n_0}$ is supported on $\tau$ and $\sigma$ if the unitary part of the Bernstein--Zelevinsky data of $\sigma_\gamma$ is contained in $\{\tau,~\tau^\vee,~\sigma\}$, i.e., $\sigma_\gamma$ is a constituent of the normalized induced representation
\[|det(\cdot)|^{s_1}\tau\times|det(\cdot)|^{s_2}\tau\times\cdots\times |det(\cdot)|^{s_t}\tau\rtimes\sigma:=Ind^{G_{n_0}}(|det(\cdot)|^{s_1}\tau\otimes|det(\cdot)|^{s_2}\tau\otimes\cdots\otimes |det(\cdot)|^{s_t}\tau\otimes\sigma) \]
for some real numbers $s_i\in \mathbb{R}$, $i=1,~\cdots,~t$. Here $\tau^\vee$ is the contragredient representation of $\tau$. Indeed, the condition that $\sigma_\gamma$ is a discrete series implies that $\tau^\vee\simeq \tau$, i.e., $\tau$ is self-dual (see \cite{tadic1998regular,heiermann2004decomposition,luo2018R}).

In the paper, we will only consider the case that $\sigma_\gamma$ is a generic discrete series representation of quasi-split $G_{n_0}$.
\subsection{Intertwining operator} For simplicity, we first exclude the case $G_n=SO_{2n}$ in what follows. For $\sigma_\gamma$ a generic discrete series representation of $G_{n_0}$ with partial cuspidal support $\sigma$ in the sense of M{\oe}glin--Tadi\'c \cite{moeglin2002construction},
we let $M(s,\tau_a,\sigma_\gamma)$ be the following standard intertwining operator
\[M(s,\tau_a,\sigma_\gamma):~\tau_a|det(\cdot)|^s\rtimes\sigma_\gamma\longrightarrow\tau_a^\vee|det(\cdot)|^{-s}\rtimes\sigma_\gamma \]
defined by the continuation of the integral \cite{waldspurger2003formule}, for $f_s(g)\in \tau_a|det(\cdot)|^s\rtimes\sigma_\gamma$,
\[\int_{N_{k}}f_s(w_kug)du. \]
Here $\tau_a^\vee\simeq (\tau^\vee)_a$, and the Weyl element $w_k\in G_n$ is given by, $N_0=2n_0$ or $2n_0+1$ depends on $G_{n_0}$,
\[w_k:=(-1)^k\begin{pmatrix}
&  & I_{k} \\ 
&  I_{N_0} &  \\ 
\pm I_{k} &  & 
\end{pmatrix}. \]
For the case $G_n=SO_{2n}$, set 
\[c=diag\left(I_{n-1}, \begin{bmatrix}
& 1 \\ 
1 & 
\end{bmatrix},I_{n-1} \right),\]
The necessary modification throughout the paper is to replace $\tau_a^\vee|det(\cdot)|^{-s}\rtimes\sigma_\gamma$ and $w_k$ by 
\[\mbox{$c^k.\left(\tau_a^\vee|det(\cdot)|^{-s}\rtimes\sigma_\gamma\right)$ and $c^kw_k$, respectively.}\]
As seen in the introduction, certain reduced decompositions of $w_k$ has been, and will be used often in our argument. For the convenience of the readers, we record them as follows.
\begin{align*}
\mbox{Way 1,2}:~&w_k=diag\left(I_{k-d},w_{d},I_{k-d}\right) \cdot diag\left(\begin{bmatrix}
& I_{k-d} \\ 
I_d & 
\end{bmatrix},I_{N_0}, \begin{bmatrix}
& I_d \\ 
I_{k-d} & 
\end{bmatrix}\right) \cdot diag\left(I_d,w_{k-d},I_d\right),\\
\mbox{Way 3,4}:~&w_k=diag\left(\begin{bmatrix}
& I_k \\ 
I_d & 
\end{bmatrix}, I_{N_0-2d}, \begin{bmatrix}
& I_d \\ 
I_k & 
\end{bmatrix}\right) diag\left(I_d,w_k,I_d\right) diag\left(\begin{bmatrix}
& I_d \\ 
I_k & 
\end{bmatrix},I_{N_0-2d},\begin{bmatrix}
& I_k \\ 
I_d & 
\end{bmatrix}\right).
\end{align*}

Following Langlands--Shahidi's normalization of intertwining operators \cite{shahidi1990proof,casselman1998irreducibility}, we define $M^*(s,\tau_a,\sigma_\gamma)$ to be
\[\frac{1}{\alpha(s,\tau_a,\sigma_\gamma)}M(s,\tau_a,\sigma_\gamma), \]
and 
\[ \alpha(s,\tau_a,\sigma_\gamma):=L(2s,\tau_a,\rho)L(s,\tau_a\times \sigma_\gamma), \]
where for $n_0=0$, $L(s,\tau\times\sigma_\gamma):=L(s,\tau)$ if $G_n=Sp_{2n}$ , 1 otherwise. Here $\rho$ (resp. $\rho^-$) is defined as follows.
\begin{align*}
\rho:&=\begin{cases}
Sym^2,& \mbox{if } G_n=SO_{2n+1},\\
\Lambda^2,& \mbox{if } G_n=Sp_{2n}\mbox{ or } SO_{2n};
\end{cases}\\
\rho^-:&=\begin{cases}
\Lambda^2,& \mbox{if } G_n=SO_{2n+1},\\
Sym^2,& \mbox{if } G_n=Sp_{2n} \mbox{ or }SO_{2n}.
\end{cases}
\end{align*}
Here $Sym^2$ (resp. $\Lambda^2$) is the symmetric (resp. exterior) second power of the standard representation of $GL_k$.

On the other hand, based on the calculation of Plancherel measure (see \cite{shahidi1990proof}), we know that, up to invertible elements in $\mathbb{C}[q^{-s},~q^s]$,
\[M^*(-s,\tau_a^\vee,\sigma_\gamma)\circ M^*(s,\tau_a,\sigma_\gamma)=\frac{1}{\beta(s,\tau_a,\sigma_\gamma)\beta(-s,~\tau_a^\vee,\sigma_\gamma)}. \]
Here $\beta(s,\tau_a,\sigma_\gamma)$ is given similarly as $\alpha(s,\tau_a,\sigma_\gamma)$ by
\[\beta(s,\tau_a,\sigma_\gamma):=L(2s+1,\tau_a,\rho)L(s+1,\tau_a\times \sigma_\gamma). \]
Now we can state our Main Theorem as follows.
\begin{mthm}(Casselman--Shahidi conjecture \cite{casselman1998irreducibility})\label{mainthm}
	Retain the notions as above. We have
	\[M^*(s,\tau_a,\sigma_\gamma):=\frac{1}{\alpha(s,\tau_a,\sigma_\gamma)}M(s,\tau_a,\sigma_\gamma)\mbox{ is holomorphic for all }s\in \mathbb{C}. \]
\end{mthm}
\begin{rem}\label{rmkreduction}
	One may see that the holomorphy of $M^*(s,\tau_a,\sigma_\gamma)$ follows readily from \cite{silberger1980special} if $\tau\not\simeq \tau^\vee$. On the other hand, it is also easy to see that we can reduce the problem to the case that $\sigma_\gamma$ is supported on $\tau$ and $\sigma$ following the argument we provided in the next section, as a by-product of a ``product formula'' decomposition pattern of induced representations in the sense of Jantzen (see \cite{jantzen1997supports,jantzenluo}). {\bf Therefore, we always assume that $\tau\simeq \tau^\vee$ (or $\tau^*$ the conjugate contragredient dual in Remark \ref{rmkforquasi-split}) and $\sigma_\gamma$ is supported on $\tau$ and $\sigma$ in the remaining of the paper}. On the other hand, the original Casselman--Shahidi conjecture is stated for arbitrary co-rank one standard modules, i.e., induced representations of essentially tempered representations as opposed to essentially discrete series representations herein, but the reduction step is an easy corollary of certain multiplicativity property of $L$-functions given in \cite[Theorem 5.1]{casselman1998irreducibility}.
\end{rem}
\begin{rem}\label{rmkforquasi-split}
	Indeed, one can apply the same formulation and argument in the paper to establish Main Theorem \ref{mainthm} for quasi-split $G_n=U_{2n+1,E/F}$, $U_{2n,E/F}$, and $SO^*_{2n}$, as well as the non-connected $O_{2n}$, with a slight modification for the unitary groups case as follows, we leave the details to the readers,
	\begin{itemize}
		\item $L(s,\tau\times\sigma)=L(s,\tau\times BC(\sigma))$, where $BC$ is the base change for unitary groups (see \cite{kimkrishnamurthy2004,kimkrishnamurthy2005,cogdellshapiroshahidi2011,mok2015endoscopic,soudrytanay2015}). For $n_0=0$, $L(s,\tau\times\sigma):=L(s,\tau)$ if $G_n=U_{2n+1}$, 1 otherwise. 
		\item $\rho$ (resp. $\rho^-$) is defined as follows.
		\begin{align*}
		\rho:&=\begin{cases}
		\mbox{Asai},& \mbox{if } G_n=U_{2n},\\
		\mbox{Asai}\otimes \chi_{E/F}, &\mbox{if } G_n=U_{2n+1};
		\end{cases}\\
		\rho^-:&=\begin{cases}
		\mbox{Asai}\otimes \chi_{E/F},& \mbox{if } G_n=U_{2n},\\
		\mbox{Asai}, &\mbox{if } G_n=U_{2n+1}.
		\end{cases}
		\end{align*}
		Here ``Asai'' is the Asai representation of the $L$-group of the scalar restriction $Res_{E/F}(GL_k)$, $\chi_{E/F}$ is the quadratic character associated to the field extension $E/F$ via class field theory.
	\end{itemize}	 
\end{rem}
With the help of Shahidi's local coefficient theory, we can obtain the following corollaries of Main Theorem \ref{mainthm}. We would like to point out that the following argument works in general, provided that Casselman--Shahidi's holomorphicity conjecture is established, for example the setting of groups of classical groups (see Theorem \ref{mthgeneral}). As seen in Remark \ref{rmkreduction}, that is to say that our corollaries stated below can be extended to the setting of arbitrary co-rank one parabolic induced representations inducing from essentially tempered representations, instead of essentially discrete series just presented in the paper for simplicity, for all quasi-split groups of classical type without any difficulty.
\begin{cor}\label{Cornonzero}
	Keep the notation as before. We have
	\[M^*(s,\tau_a,\sigma_{\gamma}) \mbox{ is always non-zero for }s\in \mathbb{C}. \]
\end{cor} 
\begin{cor}[Standard module conjecture]\label{standardmoduleconj}
	Maintain the notions as defined earlier. If the unique Langlands quotient of the standard module $|det(\cdot)|^{s}\tau_a\rtimes \sigma_{\gamma}$ with $Re(s)>0$ is generic, then
	\[|det(\cdot)|^{s}\tau_a\rtimes \sigma_{\gamma}\mbox{ is irreducible.} \] 
\end{cor}
\begin{proof}[Proof of Corollaries \ref{Cornonzero} and \ref{standardmoduleconj}]
	Recall Shahidi's local coefficient theory, given by the following diagram, please refer to \cite{shahidi1981certain} for the notation and results in details,
	\[ \xymatrix{ |det(\cdot)|^{s}\tau_a\rtimes \sigma_{\gamma}\ar[rr]^{M(s,\tau_a,\sigma_{\gamma})} \ar[dr]_{\lambda(s,\cdots)} & & |det(\cdot)|^{-s}\tau_a\rtimes \sigma_{\gamma}\ar[dl]^{\lambda(-s,\cdots)} \\
	&	\mathbb{C}_{\psi} &
	 }\]
 says that there exists a rational coefficient $C_\psi(s,\cdots)$ given by, up to an invertible element in $\mathbb{C}[q^{-s},q^s]$,
 \[C_\psi(s,\cdots)=\frac{\beta(-s,\tau_a,\sigma_{\gamma})}{\alpha(s,\tau_a,\sigma_{\gamma})} \]
 such that, up to an invertible element in $\mathbb{C}[q^{-s},q^s]$,
 \[\lambda(s,\cdots)=C_\psi(s,\cdots)\lambda(-s,\cdots)\circ M(s,\tau_a,\sigma_{\gamma})=\beta(-s,\cdots)\lambda(-s,\cdots)\circ M^*(s,\tau_a,\sigma_{\gamma}). \]
   On one hand, it is well-known that $\beta(s)$ has no poles for $Re(s)> -\frac{1}{2}$ (see \cite[Section 4]{casselman1998irreducibility} or \cite{heiermann2013tempered}). On the other hand, \cite[Proposition 3.1]{shahidi1981certain} tells us that $\lambda(s,\cdots)$ is holomorphic and always non-zero for $s\in \mathbb{C}$ (see also \cite{casselman1980unramified}). Combining those facts together, it is readily to see that \[M^*(s,\tau_a,\sigma_{\gamma})\mbox{ is always non-zero,} \]
   thus finishing the proof of Corollary \ref{Cornonzero}. Moreover, one can also see readily that, for $Re(s)<0$, the image of $M^*(s,\tau_a,\sigma_{\gamma})$ is generic, i.e.,
   \[Im(M^*(s,\tau_a,\sigma_{\gamma}))\mbox{ is generic for }Re(s)<0. \]
   Which in turn contradicts with the uniqueness property of generic constituents in a parabolic induced representation if \[|det(\cdots)|^{-s}\tau_a\rtimes \sigma_{\gamma}\mbox{ is reducible,}\] 
   whence completing the proof of Corollary \ref{standardmoduleconj}.    
\end{proof}	
\begin{rem}
	One may notice that Casselman--Shahidi's standard module conjecture, i.e., Corollary \ref{standardmoduleconj}, has been extensively investigated by Mui{\'c}--Shahidi among others, and proved in general by Heiermann--Mui{\'c} (see \cite{heiermann2007standard,muicshahidi1998standardmoduleconj}). But their approach is different from our simple argument which follows the spirit of Casselman--Shahidi in \cite{casselman1998irreducibility}.
\end{rem}

\section{proof of main theorem \ref{mainthm}}\label{sectionproof}
Before turning to the proof of our Main Theorem \ref{mainthm}, let us first prepare some lemmas which will be needed later on as follows. Recall that $\tau_a$ is a self-dual discrete series representation of $GL_k$ associated to $\tau$ a supercuspidal representation of $GL_{\frac{k}{a}}$, i.e., \[\tau_a\hookrightarrow |det(\cdot)|^{\frac{(a-1)}{2}}\tau\times|det(\cdot)|^{\frac{(a-3)}{2}}\tau\times\cdots\times |det(\cdot)|^{-\frac{(a-1)}{2}}\tau.\]
Consider the standard intertwining operator as in \cite[Part 1]{moeglin1989residue}
\[M_{GL}(s,\tau_a, \tau_b):~ |det(\cdot)|^s\tau_a\times |det(\cdot)|^{-s}\tau_b\longrightarrow |det(\cdot)|^{-s}\tau_b\times |det(\cdot)|^s\tau_a,\]
we define the associated normalized version to be
\[M^*_{GL}(s,\tau_a,\tau_b):=\frac{1}{\alpha_{GL}(s,\tau_a,\tau_b)}M_{GL}(s,\tau_a,\tau_b). \]
Here $\alpha_{GL}(s,\tau_a,\tau_b):=L(2s,\tau_a\times \tau_b)$. Then we have
\begin{lem}\label{lemGL}
	$M^*_{GL}(s,\tau_a,\tau_b)$ is holomorphic for any $s\in \mathbb{C}$.
\end{lem}
\begin{proof}
We would like to give an elementary proof using some standard properties of intertwining operators. Such a simple idea will be applied later on. 
	
\underline{(i). Holomorphicity of co-rank one operator}: $M_{GL}(s,\tau,\tau)$ has only a simple pole at $Re(s)=0$ and its image is the subrepresentation if $|det(\cdot)|^s\tau\times |det(\cdot)|^{-s}\tau$ is reducible, i.e., $s=\pm 1$. Moreover, $M_{GL}(-s,\tau,\tau)\circ M_{GL}(s,\tau,\tau)=\mu(s,\tau)^{-1}id$. Here $\mu(s,\tau)=\frac{(1-q^{-rs})(1-q^{rs})}{(1-q^{r(1-s)})(1-q^{r(1+s)})}$ for some positive integer $r$, up to a non-zero scalar, is the co-rank one Plancherel measure.	

\underline{(ii). Reduced decompositions}: viewing $\tau_a$ or $\tau_b$ as a subrepresentation, i.e., if $a>1$ and $b>1$,
\begin{align*}
\mbox{Way 1: }&\tau_a\hookrightarrow |det(\cdot)|^{\frac{(a-1)}{2}}\tau\times |det(\cdot)|^{-\frac{1}{2}}\tau_{a-1},\\
\mbox{Way 2: }&\tau_a\hookrightarrow |det(\cdot)|^{\frac{1}{2}}\tau_{a-1}\times |det(\cdot)|^{-\frac{a-1}{2}}\tau,\\
\mbox{Way 3: }&\tau_b\hookrightarrow |det(\cdot)|^{\frac{(b-1)}{2}}\tau\times |det(\cdot)|^{-\frac{1}{2}}\tau_{b-1},\\
\mbox{Way 4: }&\tau_b\hookrightarrow |det(\cdot)|^{\frac{1}{2}}\tau_{b-1}\times |det(\cdot)|^{-\frac{b-1}{2}}\tau,\\
\end{align*}
Thus we have the following corresponding commutative diagrams which are enough for our argument,

\underline{Way 2:}
\[\xymatrix{|det(\cdot)|^{s}\tau_a\times |det(\cdot)|^{-s}\tau_b  \ar@{^(->}[r] \ar[dd]_{M_{GL}(s,\tau_a,\tau_b)} & |det(\cdot)|^{s+\frac{1}{2}}\tau_{a-1}\times\mbox{$\underbrace{|det(\cdot)|^{s-\frac{a-1}{2}}\tau\times |det(\cdot)|^{-s}\tau_b}$}\ar[d]^{M_{GL}(\cdots)}\\
&\mbox{$\underbrace{|det(\cdot)|^{s+\frac{1}{2}}\tau_{a-1}\times |det(\cdot)|^{-s}\tau_b}$}\times|det(\cdot)|^{s-\frac{a-1}{2}}\tau \ar[d]^{M_{GL}(\cdots)}\\ 
|det(\cdot)|^{-s}\tau_b\times |det(\cdot)|^{s}\tau_a \ar@{^(->}[r]&|det(\cdot)|^{-s}\tau_b\times|det(\cdot)|^{s+\frac{1}{2}}\tau_{a-1}\times|det(\cdot)|^{s-\frac{a-1}{2}}\tau,
}\]

\underline{Way 3:}
\[\xymatrix{|det(\cdot)|^{s}\tau_a\times |det(\cdot)|^{-s}\tau_b  \ar@{^(->}[r] \ar[dd]_{M_{GL}(s,\tau_a,\tau_b)} & \mbox{$\underbrace{|det(\cdot)|^{s}\tau_{a}\times|det(\cdot)|^{-s+\frac{b-1}{2}}\tau}$}\times |det(\cdot)|^{-s-\frac{1}{2}}\tau_{b-1}\ar[d]^{M_{GL}(\cdots)}\\
&|det(\cdot)|^{-s+\frac{b-1}{2}}\tau\times\mbox{$\underbrace{|det(\cdot)|^{s}\tau_{a}\times|det(\cdot)|^{-s-\frac{1}{2}}\tau_{b-1}}$} \ar[d]^{M_{GL}(\cdots)}\\ 
|det(\cdot)|^{-s}\tau_b\times |det(\cdot)|^{s}\tau_a \ar@{^(->}[r] & |det(\cdot)|^{-s+\frac{b-1}{2}}\tau\times|det(\cdot)|^{-s-\frac{1}{2}}\tau_{b-1}\times|det(\cdot)|^{s}\tau_a.
}\]
It is an easy calculation to see that 
\[L(2s,\tau_a\times\tau_b)=\begin{cases}
L(2s-\frac{b-1}{2},\tau_a\times \tau)L(2s+\frac{1}{2},\tau_a\times \tau_{b-1}), & \mbox{ if }a\geq b;\\
L(2s-\frac{a-1}{2},\tau\times \tau_b)L(2s+\frac{1}{2},\tau_{a-1}\times \tau_b), & \mbox{ if }a< b.
\end{cases} \]
Which implies that the normalization factors of intertwining operators also match each other under Way 3 if $a\geq b$ or Way 2 if $a<b$. Thus it reduces to prove the holomorphy of $M^*_{GL}(s,\tau_a,\tau)$ and $M^*_{GL}(s,\tau,\tau_b)$. In what follows, we only discuss the former case, while the latter case can be proved similarly. Consider the reduced decompositions given by the above Way 2 and the following Way 1 with $a>1$ and $b=1$,

\underline{Way 1:}
\[\xymatrix{|det(\cdot)|^{s}\tau_a\times |det(\cdot)|^{-s}\tau_b  \ar@{^(->}[r] \ar[dd]_{M_{GL}(s,\tau_a,\tau_b)} & |det(\cdot)|^{s+\frac{a-1}{2}}\tau\times\mbox{$\underbrace{|det(\cdot)|^{s-\frac{1}{2}}\tau_{a-1}\times |det(\cdot)|^{-s}\tau_b}$}\ar[d]^{M_{GL}(\cdots)}\\
	&\mbox{$\underbrace{|det(\cdot)|^{s+\frac{a-1}{2}}\tau\times |det(\cdot)|^{-s}\tau_b}$}\times|det(\cdot)|^{s-\frac{1}{2}}\tau_{a-1} \ar[d]^{M_{GL}(\cdots)}\\ 
	|det(\cdot)|^{-s}\tau_b\times |det(\cdot)|^{s}\tau_a \ar@{^(->}[r]&|det(\cdot)|^{-s}\tau_b\times|det(\cdot)|^{s+\frac{a-1}{2}}\tau\times|det(\cdot)|^{s-\frac{1}{2}}\tau_{a-1},
}\] 
and compute the discrepancies $P_i$ of the normalization factors on the left hand side and the right hand side corresponding to Way i, i=1, 2, we obtain
\[P_i:=\frac{\alpha_{GL}(\cdots)\alpha_{GL}(\cdots)}{\alpha_{GL}(s,\tau_a,\tau)}=\begin{cases}
L\left(2s+\frac{a-3}{2},\tau\times \tau\right),&\mbox{ if }i=1;\\
L\left(2s-\frac{a-1}{2}, \tau\times \tau\right),&\mbox{ if }i=2.
\end{cases} \]
Therefore we know that the set of common poles of $P_1$ and $P_2$ is not empty, i.e., $(P_1,~P_2)\neq 1$, if and only if 
\[\frac{a-3}{2}=-\frac{a-1}{2},\mbox{ i.e., }a=2. \]
Which in turn says that the common pole is at $2Re(s)=\frac{1}{2}$. Now we show directly that $M^*_{GL}(s,\tau_2,\tau)$ is holomorphic at $2s=\frac{1}{2}$. This follows from the detailed information we have in the commutative diagram given by Way 1 as follows.

\[\xymatrix{|det(\cdot)|^{\frac{1}{2}}\tau_2\times \tau \ar@{^(->}[r] \ar[dd]_{M_{GL}(s,\tau_a,\tau)} & |det(\cdot)|\tau\times\mbox{$\underbrace{\tau\times \tau}$}\ar[d]^{\mbox{simple pole}}_{scalar}\\
	&\mbox{$\underbrace{|det(\cdot)|\tau\times \tau}$}\times\tau \ar[d]^{subrep.\mapsto 0}\\ 
	\tau\times |det(\cdot)|^{\frac{1}{2}}\tau_a \ar@{^(->}[r]&\tau\times|det(\cdot)|\tau\times\tau,
}\] 
As the second intertwining operator on the right hand side maps $|det(\cdot)|^{\frac{1}{2}}\tau_2\times \tau$ to zero, and the first intertwining operator is a scalar of simple pole, so the composition $M_{GL}(s,\tau_2,\tau)$ is holomorphic at $2s=\frac{1}{2}$, thus $M^*_{GL}(s,\tau_a,\tau)$ is holomorphic in $s$ by induction. Whence our Lemma holds.
\end{proof}
\begin{rem}\label{moeglin}
	One may note that the above lemma follows also easily from M{\oe}glin--Waldspurger's deep result in \cite[Part 1]{moeglin1989residue} which, in some sense, relies heavily on Shahidi's local coefficient theory. But Shahidi's local coefficient theory can't be applied directly to non-quasi-split groups, especially general linear groups over division algebras. To my understanding, the representation-theoretical aspect of normalization of intertwining operators should, in principle, have nothing to do with the Langlands--Shahidi theory, but rather the intrinsic Harish-Chandra's Plancherel measure as illustrated in the proof. However, an explicit formula for Plancherel measure would be helpful and vital to avoid some abstract arguments. Such a formula is established for generic representation and is conjectured to be preserved in an L-packet under the conjectural local Langlands correspondence by F. Shahidi (see \cite[Section 9]{shahidi1990proof} or \cite{ganichino2014}), especially it is known for general linear groups over division algebras via the so-called Jacquet--Langlands correspondence locally and globally \cite{delignekazhdanvigneras1984,badulescu2008,aubertplymen2005plancherel}. In view of the latter result, it is readily to generalize the above lemma to the case $GL_k(D)$ with $D/F$ a division algebra. On the other hand, it is also possible to generalize \cite{moeglin1989residue}, especially Part $(I)$, to the setting of $GL_k(D)$ which we leave to a future work.
\end{rem}
Next we recall the classification of generic discrete series of $G_n$ and the associated Langlands parameters (see \cite{tadic1996generic,muic1998,moeglin2002classification,moeglin2002construction,cogdellkimshapiroshahidi2004,jantzen2000squareintegrable,jantzen2000squareintegrableII,jiangsoudry2003,jiangsoudry2012,arthur2013endoscopic,jantzenliu2014}). The Langlands parameter part is only needed to see clearly the decomposition formula for the tensor product $L$-function $L(s,\tau_a\times \sigma_{\bar{r}})$ in what follows. Indeed, it could be deduced from Langlands--Shahidi's theory, especially the multiplicativity of $\gamma$-factors. With the help of Heiermann's criterion of the existence of discrete series \cite{heiermann2004decomposition}, such an approach will be adopted to attack the Casselman--Shahidi conjecture for other groups in a future work.
\begin{enumerate}[(i).]
	\item generic supercuspidal $\sigma$: $N=2n+1$ or $2n$ depends only on $G_n$,
	\[\phi_\sigma:~W_F\longrightarrow  {^{L}G_n}(\mathbb{C})\stackrel{Std}{\longrightarrow} GL_N(\mathbb{C}) \mbox{ satisfies }\]
	\begin{itemize}
		\item $\phi_\sigma=\oplus_i \phi_{\rho_i}$ with $\rho_i:~ W_F\longrightarrow GL_{N_i}(\mathbb{C})$ irreducible associated to $\rho_i$ supercuspidal representation of $GL_{N_i}$ of type ${^{L}G_n}$,
		\item $\phi_{\rho_i}\not\simeq \phi_{\rho_j}$ for any $i\neq j$.
	\end{itemize}
    \item generic discrete series $\sigma_{\bar{r}}$ supported on $\tau$ and $\sigma$: 
    \[\phi_{\sigma_{\bar{r}}}:~W_F\longrightarrow  {^{L}G_n}(\mathbb{C})\stackrel{Std}{\longrightarrow} GL_N(\mathbb{C}) \mbox{ satisfies } \]
    \begin{itemize}
    	\item $\phi_{\sigma_{\bar{r}}}=\phi_{\tau}\otimes S_{r_1}\oplus\phi_{\tau}\otimes S_{r_2}\oplus\cdots\oplus\phi_{\tau}\otimes S_{r_t}\oplus\phi_\sigma$,
    	\item $r_1>r_2>\cdots>r_t\geq -1$ of the same parity and $t$ is even,
    	\item set $\phi_{\tau}\otimes S_0=0$ and $\phi_{\tau}\otimes S_{-1}=-\phi_{\tau}\otimes S_{1}$.
    \end{itemize}
\end{enumerate} Then the local Langlands correspondence says that
\[\mbox{generic discrete series $\sigma_{\bar{r}}$ are parameterized by those Langlands parameters }\phi_{\sigma_{\bar{r}}}, i.e., \phi_{\sigma_{\bar{r}}}\leftrightarrow \sigma_{\bar{r}} \mbox{ with}\]
\[  \sigma_{\bar{r}}\hookrightarrow |det(\cdot)|^{\frac{r_1-r_2}{4}}\tau_{\frac{r_1+r_2}{2}}\times |det(\cdot)|^{\frac{r_3-r_4}{4}}\tau_{\frac{r_3+r_4}{2}}\times\cdots\times|det(\cdot)|^{\frac{r_{t-1}-r_t}{4}}\tau_{\frac{r_{t-1}+r_t}{2}}\rtimes\sigma.\tag{$\star$} \]
Meanwhile we have the following formula for the tensor product $L$-function
\[L(s,\tau_a\times\sigma_{\bar{r}})=\prod_{i=1}^{t}L(s,\tau_a\times \tau_{r_i})\cdot L(s,\tau_a\times \sigma)=\prod_{i=1}^{t}L(s,\tau_a\times \tau_{r_i})\cdot L\left(s+\frac{a-1}{2},\tau\times \sigma\right), \tag{$\star\star$}\]
where we set $L(s,\tau_a\times \tau_{0}):=1$ and $L(s,\tau_a\times \tau_{-1}):=L(s,\tau_a\times \tau)^{-1}$.

For the convenience of the readers, we summarize some formulas which are applied often for the calculations carried out in the proof of Main Theorem \ref{mainthm} in what follows.
\[\begin{cases}
	L(2s,\tau_a,\rho)=\prod\limits_{i=1}^{\lceil\frac{a}{2}\rceil}L(2s+a+1-2i,\tau,\rho)\prod\limits_{i=1}^{\lfloor\frac{a}{2}\rfloor}L(2s+a-2i,\tau,\rho^-),&\\
	L(2(s-\frac{1}{2}),\tau_{a-1},\rho)=\prod\limits_{i=2}^{\lceil\frac{a+1}{2}\rceil}L(2s+a+1-2i,\tau,\rho)\prod\limits_{i=2}^{\lfloor\frac{a+1}{2}\rfloor}L(2s+a-2i,\tau,\rho^-),&\\
	L(2(s+\frac{1}{2}),\tau_{a-1},\rho)=\prod\limits_{i=1}^{\lceil\frac{a-1}{2}\rceil}L(2s+a+1-2i,\tau,\rho)\prod\limits_{i=1}^{\lfloor\frac{a-1}{2}\rfloor}L(2s+a-2i,\tau,\rho^-).&
\end{cases} \tag{A}\]
\[ L(s,\tau_a\times \tau_{r})=\begin{cases}
	\prod\limits_{i=-\frac{r-1}{2}}^{\frac{r-1}{2}}L(s+\frac{a-1}{2}+i,\tau\times \tau),&\mbox{ if }a\geq r;\\
	\prod\limits_{i=-\frac{a-1}{2}}^{\frac{a-1}{2}}L(s+\frac{r-1}{2}+i,\tau\times \tau),&\mbox{ if }a< r.
	\end{cases} \tag{B}\] 
Now back to our proof of our Main Theorem \ref{mainthm}. Our argument uses the intrinsic non-symmetry property of normalization factors corresponding to different reduced decompositions of Weyl elements as seen in Lemma \ref{lemGL}. 
\begin{proof}[Proof of Main Theorem \ref{mainthm}]
The proof involves a reduction step and an induction argument as follows.

	\underline{\bf Step 1 }(Reduction step). We first reduce the general case $\sigma_{\bar{r}}$ to the case $\sigma_r$ associated to a segment, i.e.,
	\[\phi_{\sigma_{r}}\leftrightarrow \phi_{\tau}\otimes S_{r_1}\oplus \phi_{\tau}\otimes S_{r_2}\oplus\phi_\sigma\mbox{ with }r_1>a>r_2. \]
	This follows from the following facts.
	\begin{enumerate}[(a).]
		\item For $a\geq r_1>r_2$ or $r_1>r_2\geq a$, we have the identity of $L$-functions.
		\[L\left(s-\frac{r_1-r_2}{4},\tau_a\times \tau_{\frac{r_1+r_2}{2}}\right)L\left(s+\frac{r_1+r_2}{4},\tau_a\times\tau_{\frac{r_1+r_2}{2}}\right)=L(s,\tau_a\times \tau_{r_1})L(s,\tau_a\times\tau_{r_2}).\]
		\item For $r_1>r_2>\cdots>r_t\geq -1$ of the same parity and $t$ is even, we know the induced representation 
		\[|det(\cdot)|^{\frac{r_1-r_2}{4}}\tau_{\frac{r_1+r_2}{2}}\times |det(\cdot)|^{\frac{r_3-r_4}{4}}\tau_{\frac{r_3+r_4}{2}}\times\cdots\times|det(\cdot)|^{\frac{r_{t-1}-r_t}{4}}\tau_{\frac{r_{t-1}+r_t}{2}} \]
		is irreducible (see \cite{bernstein1977induced,zelevinsky1980induced}).
	\end{enumerate}
To be precise, in light of ($\star$), viewing $\sigma_{\bar{r}}\hookrightarrow |det(\cdot)|^{\frac{r_1-r_2}{4}}\tau_{\frac{r_1+r_2}{2}}\rtimes \sigma_{\bar{r'}}$ with
\[\sigma_{\bar{r'}}\leftrightarrow\phi_{\sigma_{\bar{r'}}}=\phi_\tau\otimes S_{r_3}\oplus\phi_\tau\otimes S_{r_4}\oplus\cdots\oplus\phi_\tau\otimes S_{r_t}\oplus\phi_\sigma. \]
We have the following decomposition of $M(s,\tau_a,\sigma_{\bar{r}})$:
\[\xymatrix{|det(\cdot)|^{s}\tau_a\rtimes \sigma_{\bar{r}} \ar@{^(->}[r] \ar[ddd]_{M(s,\tau_a,\sigma_{\bar{r}})} & \mbox{$\underbrace{|det(\cdot)|^{s}\tau_{a}\times|det(\cdot)|^{\frac{r_1-r_2}{4}}\tau_{\frac{r_1+r_2}{2}}}$}\rtimes \sigma_{\bar{r'}}\ar[d]^{M_{GL}(\cdots)}\\
	&|det(\cdot)|^{\frac{r_1-r_2}{4}}\tau_{\frac{r_1+r_2}{2}}\times\mbox{$\underbrace{|det(\cdot)|^{s}\tau_{a}\rtimes\sigma_{\bar{r'}}}$} \ar[d]^{M(s,\tau_a,\sigma_{\bar{r'}})}\\ 
	&\mbox{$\underbrace{|det(\cdot)|^{\frac{r_1-r_2}{4}}\tau_{\frac{r_1+r_2}{2}}\times|det(\cdot)|^{-s}\tau_a}$}\rtimes \sigma_{\bar{r'}}  \ar[d]^{M_{GL}(\cdots)}\\
	|det(\cdot)|^{-s}\tau_a\rtimes \sigma_{\bar{r}} \ar@{^(->}[r] & |det(\cdot)|^{-s}\tau_a\times|det(\cdot)|^{\frac{r_1-r_2}{4}}\tau_{\frac{r_1+r_2}{2}}\rtimes\sigma_{\bar{r'}}.
}\]
Via Lemma \ref{lemGL} + ($\star\star$) + (A)(B) + (a)(b), one can do the easy calculation of normalization factors of those intertwining operators and see that they also match each other, thus the reduction step holds.

\underline{\bf Step 2 }(Initial step for induction). There are two initial cases:
\[\mbox{Case }|det(\cdot)|^{s}\tau_a\rtimes\sigma,\mbox{  and Case }|det(\cdot)|^s\tau\rtimes \sigma_{r}. \]

\underline{The case $|det(\cdot)|^{s}\tau_a\rtimes \sigma$ with $a>1$}: We have two ways to decompose $M(s,\tau_a,\sigma)$ and there are two corresponding commutative diagrams as follows.

\underline{Way 1}: Viewing $\tau_a\hookrightarrow|det(\cdot)|^{\frac{a-1}{2}}\tau\times |det(\cdot)|^{-\frac{1}{2}}\tau_{a-1}$, it gives rise to
\[\xymatrix{|det(\cdot)|^{s}\tau_a\rtimes \sigma \ar@{^(->}[r] \ar[ddd]_{M(s,\tau_a,\sigma)} & |det(\cdot)|^{s+\frac{a-1}{2}}\tau\times\mbox{$\underbrace{|det(\cdot)|^{s-\frac{1}{2}}\tau_{a-1}\rtimes \sigma}$}\ar[d]^{M(s-\frac{1}{2},\tau_{a-1},\sigma)}\\
	&\mbox{$\underbrace{|det(\cdot)|^{s+\frac{a-1}{2}}\tau\times|det(\cdot)|^{-s+\frac{1}{2}}\tau_{a-1}}$}\rtimes \sigma  \ar[d]^{M_{GL}(\cdots)}\\
	&|det(\cdot)|^{-s+\frac{1}{2}}\tau_{a-1}\times\mbox{$\underbrace{|det(\cdot)|^{s+\frac{a-1}{2}}\tau\rtimes\sigma}$} \ar[d]^{M(s+\frac{a-1}{2},\tau,\sigma)}\\ 
	|det(\cdot)|^{-s}\tau_a\rtimes \sigma \ar@{^(->}[r] & |det(\cdot)|^{-s+\frac{1}{2}}\tau_{a-1}\times|det(\cdot)|^{-s-\frac{a-1}{2}}\tau\rtimes\sigma.
}\]

\underline{Way 2}: Viewing $\tau_a\hookrightarrow|det(\cdot)|^{\frac{1}{2}}\tau_{a-1}\times |det(\cdot)|^{-\frac{a-1}{2}}\tau$, it gives rise to
\[\xymatrix{|det(\cdot)|^{s}\tau_a\rtimes \sigma \ar@{^(->}[r] \ar[ddd]_{M(s,\tau_a,\sigma)} & |det(\cdot)|^{s+\frac{1}{2}}\tau_{a-1}\times\mbox{$\underbrace{|det(\cdot)|^{s-\frac{a-1}{2}}\tau\rtimes \sigma}$}\ar[d]^{M(s-\frac{a-1}{2},\tau,\sigma)}\\
	&\mbox{$\underbrace{|det(\cdot)|^{s+\frac{1}{2}}\tau_{a-1}\times|det(\cdot)|^{-s+\frac{a-1}{2}}\tau}$}\rtimes \sigma  \ar[d]^{M_{GL}(\cdots)}\\
	&|det(\cdot)|^{-s+\frac{a-1}{2}}\tau\times\mbox{$\underbrace{|det(\cdot)|^{s+\frac{1}{2}}\tau_{a-1}\rtimes\sigma}$} \ar[d]^{M(s+\frac{1}{2},\tau_{a-1},\sigma)}\\ 
	|det(\cdot)|^{-s}\tau_a\rtimes \sigma \ar@{^(->}[r] & |det(\cdot)|^{-s+\frac{a-1}{2}}\tau\times|det(\cdot)|^{-s-\frac{1}{2}}\tau_{a-1}\rtimes\sigma.
}\]
Via Lemma \ref{lemGL} + (A)(B), one can calculate the normalizations factors of intertwining operators and obtain their discrepancies $P_i$ associated to Way i, i=1, 2 as follows.
\begin{align*}
P_1:&=\frac{\alpha(s-\frac{1}{2},\tau_{a-1},\sigma)\alpha_{GL}(\cdots)\alpha(s+\frac{a-1}{2},\tau,\sigma)}{\alpha(s,\tau_a,\sigma)}\\
&=\begin{cases}
L(2s-1,\tau,\rho^-)L(2s+a-2,\tau,\rho)L(s+\frac{a-3}{2},\tau\times \sigma), & \mbox{ $a$ odd};\\
L(2s-1,\tau,\rho)L(2s+a-2,\tau,\rho)L(s+\frac{a-3}{2},\tau\times \sigma), & \mbox{ $a$ even}.
\end{cases}\\
P_2:&=\frac{\alpha(s-\frac{a-1}{2},\tau,\sigma)\alpha_{GL}(\cdots)\alpha(s+\frac{1}{2},\tau_{a-1},\sigma)}{\alpha(s,\tau_a,\sigma)}\\
&=\begin{cases}
L(2s,\tau,\rho^-)L(2s-(a-1),\tau,\rho)L(s-\frac{a-1}{2},\tau\times \sigma), & \mbox{ $a$ odd};\\
L(2s,\tau,\rho)L(2s-(a-1),\tau,\rho)L(s-\frac{a-1}{2},\tau\times \sigma), & \mbox{ $a$ even}.
\end{cases}
\end{align*}
Then we can first compare the possible poles of $P_1$ with poles of $P_2$ and obtain their common possible poles as follows.
\[\mbox{ Possible poles }P_1:~s=\frac{1}{2},~ -\frac{a-2}{2},~ -\frac{a-3}{2};\qquad P_2:~s=0,~\frac{a-1}{2}. \]
The common possible poles are at $s=0$ if $a=3$, and if $a=2$,
\[s=0 \mbox{ or }s=\frac{1}{2}, \]
Now we prove the holomorphy of $M^*(s,\tau_2,\sigma)$ at $s=0,~\frac{1}{2}$ directly as follows. The case $s=0$ follows easily from the fact that $M^*(0,\tau_2,\sigma)^2=id.$, up to a non-zero scalar, and the fact that $\tau_2\rtimes \sigma$ is multiplicity-free (see \cite[P. 272 Remark]{luo2020knapp} in general). It remains to discuss the case $s=\frac{1}{2}$, this is an easy corollary of the following commutative diagram
\[\xymatrix{|det(\cdot)|^{\frac{1}{2}}\tau_2\rtimes \sigma \ar@{^(->}[r] \ar[ddd]_{M(s,\tau_2,\sigma)} & |det(\cdot)|^{1}\tau\times\mbox{$\underbrace{|det(\cdot)|^{0}\tau\rtimes \sigma}$}\ar[d]^{M(0,\tau,\sigma)}_{\mbox{simple pole}}\\
	&\mbox{$\underbrace{|det(\cdot)|^{1}\tau\times|det(\cdot)|^{0}\tau}$}\rtimes \sigma  \ar[d]^{M_{GL}(\cdots)}_{subrep.\mapsto 0}\\
	&|det(\cdot)|^{0}\tau\times\mbox{$\underbrace{|det(\cdot)|^{1}\tau\rtimes\sigma}$} \ar[d]^{M(1,\tau,\sigma)}_{\mbox{holomorphic}}\\ 
	|det(\cdot)|^{-\frac{1}{2}}\tau_2\rtimes \sigma \ar@{^(->}[r] & |det(\cdot)|^{0}\tau\times|det(\cdot)|^{-1}\tau\rtimes\sigma.
}\]
To be precise, as the intertwining operator $M(0,\tau,\sigma)$ is a scalar of simple pole. Thus the composition of the first two arrows on the right hand side is holomorphic restricting to $|det(\cdot)|^{\frac{1}{2}}\tau_2\rtimes \sigma$, hence $M(s,\tau_2,\sigma)$ is holomorphic at $Re(s)=\frac{1}{2}$. Whence we finish the proof of this case by induction.

\underline{The case $|det(\cdot)|^{s}\tau\rtimes \sigma_r$ with $r_1>1>r_2$}: In this case, we have $r_2=0$ if $r_1$ is even, $-1$ otherwise. Then we know that 
\[\sigma_r\hookrightarrow |det(\cdot)|^{\frac{r_1}{4}}\tau_{\frac{r_1}{2}}\rtimes\sigma\mbox{ (if $r_1$ even)}; \quad \sigma_r\hookrightarrow |det(\cdot)|^{\frac{r_1+1}{4}}\tau_{\frac{r_1-1}{2}}\rtimes\sigma\mbox{ (if $r_1$ odd)},\]
which implies that \[\alpha(s,\tau,\sigma_r)=L(2s,\tau,\rho)L(s,\tau\times \tau_{r_1})=L(2s,\tau,\rho)L\left(s+\frac{r_1-1}{2},\tau\times \tau\right).\]
Note that we have two ways to decompose $M(s,\tau,\sigma_r)$ and there are two corresponding commutative diagrams as follows.

\underline{Way 3}: Viewing $\sigma_r\hookrightarrow|det(\cdot)|^{\frac{r_1-1}{2}}\tau\times |det(\cdot)|^{-\frac{1}{2}}\sigma_{r'}$ with $r'_1=r_1-2$, it gives rise to
\[\xymatrix{|det(\cdot)|^{s}\tau\rtimes \sigma_{r} \ar@{^(->}[r] \ar[ddd]_{M(s,\tau,\sigma_{r})} & \mbox{$\underbrace{|det(\cdot)|^{s}\tau\times|det(\cdot)|^{\frac{r_1-1}{2}}\tau}$}\rtimes \sigma_{r'}\ar[d]^{M_{GL}(\cdots)}\\
	&|det(\cdot)|^{\frac{r_1-1}{2}}\tau\times\mbox{$\underbrace{|det(\cdot)|^{s}\tau\rtimes\sigma_{r'}}$} \ar[d]^{M(s,\tau,\sigma_{r'})}\\ 
	&\mbox{$\underbrace{|det(\cdot)|^{\frac{r_1-1}{2}}\tau\times|det(\cdot)|^{-s}\tau}$}\rtimes \sigma_{r'}  \ar[d]^{M_{GL}(\cdots)}\\
	|det(\cdot)|^{-s}\tau\rtimes \sigma_{r} \ar@{^(->}[r] & |det(\cdot)|^{-s}\tau\times|det(\cdot)|^{\frac{r_1-1}{2}}\tau\rtimes\sigma_{r'}.
}\]

\underline{Way 4}: Viewing $\sigma_r\hookrightarrow|det(\cdot)|^{\frac{r_1-r_2}{4}}\tau_{\frac{r_1+r_2}{2}}\rtimes \sigma$, it gives rise to
\[\xymatrix{|det(\cdot)|^{s}\tau\rtimes \sigma_{r} \ar@{^(->}[r] \ar[ddd]_{M(s,\tau,\sigma_{r})} & \mbox{$\underbrace{|det(\cdot)|^{s}\tau\times|det(\cdot)|^{\frac{r_1-r_2}{4}}\tau_{\frac{r_1+r_2}{2}}}$}\rtimes \sigma\ar[d]^{M_{GL}(\cdots)}\\
	&|det(\cdot)|^{\frac{r_1-r_2}{4}}\tau_{\frac{r_1+r_2}{2}}\times\mbox{$\underbrace{|det(\cdot)|^{s}\tau\rtimes\sigma}$} \ar[d]^{M(s,\tau,\sigma)}\\ 
	&\mbox{$\underbrace{|det(\cdot)|^{\frac{r_1-r_2}{4}}\tau_{\frac{r_1+r_2}{2}}\times|det(\cdot)|^{-s}\tau}$}\rtimes \sigma  \ar[d]^{M_{GL}(\cdots)}\\
	|det(\cdot)|^{-s}\tau\rtimes \sigma_{r} \ar@{^(->}[r] & |det(\cdot)|^{-s}\tau\times|det(\cdot)|^{\frac{r_1-r_2}{4}}\tau_{\frac{r_1+r_2}{2}}\rtimes\sigma.
}\]
Via Lemma \ref{lemGL} + (A)(B), one can calculate the normalizations factors of intertwining operators and obtain their discrepancies $P_i$ associated to Way i, i=3, 4 as follows.
\begin{align*}
P_3:&=\frac{\alpha_{GL}(\cdots)\alpha(s,\tau,\sigma_{r'})\alpha_{GL}(\cdots)}{\alpha(s,\tau,\sigma_r)}\\
&=L\left(s-\frac{r_1-1}{2},\tau\times \tau\right)L\left(s+\frac{r_1'-1}{2},\tau\times \tau\right).\\
P_4:&=\frac{\alpha_{GL}(\cdots)\alpha(s,\tau,\sigma)\alpha_{GL}(\cdots)}{\alpha(s,\tau,\sigma_r)}\\
&=L(s,\tau\times \sigma)L\left(s+\frac{r_2-1}{2},\tau\times \tau\right).
\end{align*}
Then we can first compare the possible poles of $P_3$ with poles of $P_4$ and obtain their common possible poles as follows.
\[\mbox{ Possible poles }P_3:~s=\frac{r_1-1}{2},~ -\frac{r_1-3}{2};\qquad P_4:~s=0,~-\frac{r_2-1}{2}. \]
The common possible poles are at, $s=\frac{1}{2}$ if $r_1=2$, and if $r_1=3$,
\[s=0 \mbox{ or }s=1. \] 
Now we prove the holomorphy of $M^*(s,\tau,\sigma_r)$ at $s=0,~\frac{1}{2}$ and $1$ directly as follows. They are an easy corollary of the following commutative diagrams, 
\[(s=0, ~r_1=3, ~\mbox{ and }r_2=-1):~\xymatrix{|det(\cdot)|^{0}\tau\rtimes \sigma_{r} \ar@{^(->}[r] \ar[ddd]_{M(s,\tau,\sigma_{r})} & \mbox{$\underbrace{|det(\cdot)|^{0}\tau\times|det(\cdot)|^{1}\tau}$}\rtimes \sigma\ar[d]^{M_{GL}(\cdots)}_{subrep.\mapsto 0}\\
	&|det(\cdot)|^{1}\tau\times\mbox{$\underbrace{|det(\cdot)|^{0}\tau\rtimes\sigma}$} \ar[d]^{M(s,\tau,\sigma)}_{\mbox{simple pole}}\\ 
	&\mbox{$\underbrace{|det(\cdot)|^{1}\tau\times|det(\cdot)|^{0}\tau}$}\rtimes \sigma  \ar[d]^{M_{GL}(\cdots)}_{subrep. \mapsto 0}\\
	|det(\cdot)|^{0}\tau\rtimes \sigma_{r} \ar@{^(->}[r] & |det(\cdot)|^{0}\tau\times|det(\cdot)|^{1}\tau\rtimes\sigma.
}\]
\[(s=1,~r_1=3,\mbox{ and }r_2=-1):~\xymatrix{|det(\cdot)|^{1}\tau\rtimes \sigma_{r} \ar@{^(->}[r] \ar[ddd]_{M(s,\tau,\sigma_{r})} & \mbox{$\underbrace{|det(\cdot)|^{1}\tau\times|det(\cdot)|^{1}\tau}$}\rtimes \sigma\ar[d]^{M_{GL}(\cdots)}_{\mbox{simple pole}}\\
	&|det(\cdot)|^{1}\tau\times\mbox{$\underbrace{|det(\cdot)|^{1}\tau\rtimes\sigma}$} \ar[d]^{M(s,\tau,\sigma)}_{subrep.\mapsto 0}\\ 
	&\mbox{$\underbrace{|det(\cdot)|^{1}\tau\times|det(\cdot)|^{-1}\tau}$}\rtimes \sigma  \ar[d]^{M_{GL}(\cdots)}_{\mbox{holomorphic}}\\
	|det(\cdot)|^{-1}\tau\rtimes \sigma_{r} \ar@{^(->}[r] & |det(\cdot)|^{-1}\tau\times|det(\cdot)|^{1}\tau\rtimes\sigma.
}\]
\[(s=\frac{1}{2},~r_1=2\mbox{ and }r_2=0):~\xymatrix{|det(\cdot)|^{\frac{1}{2}}\tau\rtimes \sigma_{r} \ar@{^(->}[r] \ar[ddd]_{M(s,\tau,\sigma_{r})} & \mbox{$\underbrace{|det(\cdot)|^{\frac{1}{2}}\tau\times|det(\cdot)|^{\frac{1}{2}}\tau}$}\rtimes \sigma\ar[d]^{M_{GL}(\cdots)}_{\mbox{simple pole}}\\
	&|det(\cdot)|^{\frac{1}{2}}\tau\times\mbox{$\underbrace{|det(\cdot)|^{\frac{1}{2}}\tau\rtimes\sigma}$} \ar[d]^{M(s,\tau,\sigma)}_{subrep.\mapsto 0}\\ 
	&\mbox{$\underbrace{|det(\cdot)|^{\frac{1}{2}}\tau\times|det(\cdot)|^{-\frac{1}{2}}\tau}$}\rtimes \sigma  \ar[d]^{M_{GL}(\cdots)}_{\mbox{holomorphic}}\\
	|det(\cdot)|^{-\frac{1}{2}}\tau\rtimes \sigma_{r} \ar@{^(->}[r] &|det(\cdot)|^{-\frac{1}{2}}\tau\times|det(\cdot)|^{\frac{1}{2}}\tau\rtimes\sigma.
}\]
To be precise, the first diagram says that $M(0,\tau,\sigma_r)$ is a non-zero scalar, the second and last diagrams imply that the compositions of the first two arrows on the right hand are holomorphic restricting to $|det(\cdot)|^1\tau\rtimes \sigma_r$ and $|det(\cdot)|^{\frac{1}{2}}\rtimes \sigma_r$ respectively, thus completing the proof via induction.

\underline{\bf Step 3 }(Induction step). Note that $r_1>a>r_2$ and $a>1$, we have three ways of reduced decompositions of $M(s,\tau_a,\sigma_{r})$ corresponding to the following three embeddings.

\underline{Way 1: $\tau_a\hookrightarrow|det(\cdot)|^{\frac{a-1}{2}}\tau\times |det(\cdot)|^{-\frac{1}{2}}\tau_{a-1}$}:
\[\xymatrix{ |det(\cdot)|^{s}\tau_a\rtimes \sigma_r \ar@{^(->}[r] \ar[ddd]_{M(s,\tau_a,\sigma_r)} & |det(\cdot)|^{s+\frac{a-1}{2}}\tau\times\mbox{$\underbrace{|det(\cdot)|^{s-\frac{1}{2}}\tau_{a-1}\rtimes \sigma_r}$}\ar[d]^{M(s-\frac{1}{2},\tau_{a-1},\sigma_r)}\\
	&\mbox{$\underbrace{|det(\cdot)|^{s+\frac{a-1}{2}}\tau\times|det(\cdot)|^{-s+\frac{1}{2}}\tau_{a-1}}$}\rtimes \sigma_r  \ar[d]^{M_{GL}(\cdots)}\\
	&|det(\cdot)|^{-s+\frac{1}{2}}\tau_{a-1}\times\mbox{$\underbrace{|det(\cdot)|^{s+\frac{a-1}{2}}\tau\rtimes\sigma_r}$} \ar[d]^{M(s+\frac{a-1}{2},\tau,\sigma_r)}\\ 
	|det(\cdot)|^{-s}\tau_a\rtimes \sigma_r \ar@{^(->}[r] & |det(\cdot)|^{-s+\frac{1}{2}}\tau_{a-1}\times|det(\cdot)|^{-s-\frac{a-1}{2}}\tau\rtimes\sigma_r.
}\] 

\underline{Way 2: $\tau_a\hookrightarrow|det(\cdot)|^{\frac{1}{2}}\tau_{a-1}\times |det(\cdot)|^{-\frac{a-1}{2}}\tau$}:
\[\xymatrix{ |det(\cdot)|^{s}\tau_a\rtimes \sigma_r \ar@{^(->}[r] \ar[ddd]_{M(s,\tau_a,\sigma_r)} & |det(\cdot)|^{s+\frac{1}{2}}\tau_{a-1}\times\mbox{$\underbrace{|det(\cdot)|^{s-\frac{a-1}{2}}\tau\rtimes \sigma_r}$}\ar[d]^{M(s-\frac{a-1}{2},\tau,\sigma_r)}\\
	&\mbox{$\underbrace{|det(\cdot)|^{s+\frac{1}{2}}\tau_{a-1}\times|det(\cdot)|^{-s+\frac{a-1}{2}}\tau}$}\rtimes \sigma_r  \ar[d]^{M_{GL}(\cdots)}\\
	&|det(\cdot)|^{-s+\frac{a-1}{2}}\tau\times\mbox{$\underbrace{|det(\cdot)|^{s+\frac{1}{2}}\tau_{a-1}\rtimes\sigma_r}$} \ar[d]^{M(s+\frac{1}{2},\tau_{a-1},\sigma_r)}\\ 
	|det(\cdot)|^{-s}\tau_a\rtimes \sigma_r \ar@{^(->}[r] & |det(\cdot)|^{-s+\frac{a-1}{2}}\tau\times|det(\cdot)|^{-s-\frac{1}{2}}\tau_{a-1}\rtimes\sigma_r.
}\] 

\underline{Way 3: $\sigma_\gamma\hookrightarrow|det(\cdot)|^{\frac{r-1}{4}}\tau\rtimes \sigma_{r'}$}:
\[\xymatrix{|det(\cdot)|^{s}\tau_a\rtimes \sigma_{r} \ar@{^(->}[r] \ar[ddd]_{M(s,\tau_a,\sigma_{r})} & \mbox{$\underbrace{|det(\cdot)|^{s}\tau_{a}\times|det(\cdot)|^{\frac{r-1}{4}}\tau}$}\rtimes \sigma_{r'}\ar[d]^{M_{GL}(\cdots)}\\
	&|det(\cdot)|^{\frac{r_1-1}{4}}\tau\times\mbox{$\underbrace{|det(\cdot)|^{s}\tau_{a}\rtimes\sigma_{r'}}$} \ar[d]^{M(s,\tau_a,\sigma_{r'})}\\ 
	&\mbox{$\underbrace{|det(\cdot)|^{\frac{r_1-1}{4}}\tau\times|det(\cdot)|^{-s}\tau_a}$}\rtimes \sigma_{r'}  \ar[d]^{M_{GL}(\cdots)}\\
	|det(\cdot)|^{-s}\tau_a\rtimes \sigma_{r} \ar@{^(->}[r] & |det(\cdot)|^{-s}\tau_a\times|det(\cdot)|^{\frac{r_1-1}{4}}\tau\rtimes\sigma_{r'}.
}\]
Via Lemma \ref{lemGL} + (A)(B) + ($\star\star$), one can carry out the calculation for the corresponding normalization factors and their discrepancies $P_i$ associated to Way i, i=1, 2, 3, we have
\begin{align*}
P_1:&=\frac{\alpha\left(s-\frac{1}{2},\tau_{a-1},\sigma_{r}\right)\alpha_{GL}\left(s,\tau|det(\cdot)|^{\frac{a-1}{2}},\tau_{a-1}|det(\cdot)|^{\frac{1}{2}}\right)\alpha\left(s+\frac{a-1}{2},\tau,\sigma_{r}\right)}{\alpha(s,\tau_a,\sigma_{r})}\\
&=\begin{cases}
L(2s-1,\tau,\rho^-)L(2s+a-2,\tau,\rho)L\left(s+\frac{a-r_2}{2}-1,\tau\times \tau\right)L\left(s+\frac{a-3}{2},\tau\times \sigma\right), & \mbox{$a$ odd and $r_2>0$};\\
L(2s-1,\tau,\rho)L(2s+a-2,\tau,\rho)L\left(s+\frac{a-r_2}{2}-1,\tau\times \tau\right)L\left(s+\frac{a-3}{2},\tau\times \sigma\right), & \mbox{$a$ even and $r_2>0$};\\
L(2s-1,\tau,\rho^-)L(2s+a-2,\tau,\rho), & \mbox{$a$ odd and $r_2\leq 0$};\\
L(2s-1,\tau,\rho)L(2s+a-2,\tau,\rho), & \mbox{$a$ even and $r_2\leq 0$}.
\end{cases}
\end{align*}
Similarly,
\begin{align*}
P_2:&=\begin{cases}
L(2s,\tau,\rho^-)L(2s-(a-1),\tau,\rho)L\left(s-\frac{a-r_2}{2},\tau\times \tau\right)L\left(s-\frac{a-1}{2},\tau\times \sigma\right), & \mbox{ $a$ odd and $r_2>0$};\\
L(2s,\tau,\rho)L(2s-(a-1),\tau,\rho)L\left(s-\frac{a-r_2}{2},\tau\times \tau\right)L\left(s-\frac{a-1}{2},\tau\times \sigma\right), & \mbox{ $a$ even and $r_2>0$};\\
L(2s,\tau,\rho^-)L(2s-(a-1),\tau,\rho), & \mbox{ $a$ odd and $r_2\leq 0$};\\
L(2s,\tau,\rho)L(2s-(a-1),\tau,\rho), & \mbox{ $a$ even and $r_2\leq 0$}.
\end{cases}\\
P_3:&=\begin{cases}
L\left(s-\frac{r_1-a}{2},\tau\times \tau\right)L\left(s+\frac{r_1-a}{2}-1,\tau\times \tau\right),& \qquad a\leq r_1-2;\\
L\left(s-\frac{1}{2},\tau\times \tau\right), & \qquad a=r_1-1.
\end{cases}
\end{align*}
Thus it is easy to see the possible poles of $P_i$ are as follows.
\begin{align*}
P_1: ~&s=\frac{1}{2},~ -\frac{a-2}{2},~ -\frac{a-r_2-2}{2}, ~-\frac{a-3}{2};\\
P_2: ~&s=0, ~\frac{a-1}{2}, ~\frac{a-r_2}{2};\\
P_3: ~&s=\frac{1}{2} \mbox{ (if $a=r_1-1$)},\\
P_3: ~&s= \frac{r_1-a}{2},~ -\frac{r_1-a-2}{2} \mbox{ (if $a\leq r_1-2$)}.
\end{align*}
It is readily to see that the possible common poles are at 
\[s=0; \qquad s=\frac{1}{2} \mbox{ (if $a=r_1-1=r_2+1$ or $a=r_1-1=2$)}. \]
Now we prove the holomorphy of $M^*(s,\tau_a,\sigma_r)$ at those special cases as follows.

\underline{s=0}: Note that $M^*(s,\tau_a,\sigma_{r})\circ M^*(-s,\tau_a,\sigma_{r})=\beta(s,\tau_a,\sigma_{r})^{-1}\beta(-s,\tau_a,\sigma_{r})^{-1}id.$, up to a non-zero scalar. Here $\beta(s,\tau_a,\sigma_{r})=L(2s+1,\tau_a,\rho)L(s+1,\tau_a\times \sigma_{r})$. On one hand, $\beta(s,\tau_a,\sigma_r)$ has no poles at $Re(s)=0$. Thus $M^*(0,\tau_a,\sigma_{r})^2=id.$, up to a non-zero scalar. On the other hand, we know that $\tau_a\rtimes \sigma_{r}$ is multiplicity-free, then we must have
\[M^*(s,\tau_a,\sigma_{r})\mbox{ is holomorphic at } Re(s)=0 \mbox{ and non-zero.} \]

\underline{s=$\frac{1}{2}$}: We first discuss the case $a=2=r_1-1$, this implies that $r_2=-1$, which in turn says that \[\sigma_r\hookrightarrow |det(\cdot)|^1\tau\rtimes\sigma.\]
Now we investigate the reduced decomposition of $M(s,\tau_2,\sigma_r)$ with respect to the embedding
\[\tau_2\hookrightarrow |det(\cdot)|^{\frac{1}{2}}\tau\times|det(\cdot)|^{-\frac{1}{2}}\tau,\]
which gives us the following commutative diagram
\[\xymatrix{|det(\cdot)|^{\frac{1}{2}}\tau_2\rtimes \sigma_r \ar@{^(->}[r] \ar[ddd]_{M(s,\tau_2,\sigma_r)} & |det(\cdot)|^{1}\tau\times\mbox{$\underbrace{|det(\cdot)|^{0}\tau\rtimes \sigma_r}$}\ar[d]^{M(0,\tau,\sigma_r)}_{\mbox{simple pole}}\\
	&\mbox{$\underbrace{|det(\cdot)|^{1}\tau\times|det(\cdot)|^{0}\tau}$}\rtimes \sigma_r  \ar[d]^{M_{GL}(\cdots)}_{subrep.\mapsto 0}\\
	&|det(\cdot)|^{0}\tau\times\mbox{$\underbrace{|det(\cdot)|^{1}\tau\rtimes\sigma_r}$} \ar[d]^{M(1,\tau,\sigma_r)}_{\mbox{holomorphic}}\\ 
	|det(\cdot)|^{-\frac{1}{2}}\tau_2\rtimes \sigma_r \ar@{^(->}[r] & |det(\cdot)|^{0}\tau\times|det(\cdot)|^{-1}\tau\rtimes\sigma_r.
}\]
As $\tau\rtimes \sigma_r$ is multiplicity-free, so $M(0,\tau,\sigma_r)$ is a scalar with a simple pole. Then the diagram tells us that the composition of the first two arrows on the right hand side is holomorphic restricting to $|det(\cdot)|^{\frac{1}{2}}\rtimes \sigma_r.$ Whence this case is settled.

For the case $a=r_1-1=r_2+1$, we have the embedding $\sigma_r\hookrightarrow |det(\cdot)|^\frac{1}{2}\tau_a\rtimes \sigma$ and the following associated commutative diagram
\[\xymatrix{|det(\cdot)|^{\frac{1}{2}}\tau_a\rtimes \sigma_{r} \ar@{^(->}[r] \ar[ddd]_{M(s,\tau_a,\sigma_{\bar{r}})} & \mbox{$\underbrace{|det(\cdot)|^{\frac{1}{2}}\tau_{a}\times|det(\cdot)|^{\frac{1}{2}}\tau_{a}}$}\rtimes \sigma\ar[d]^{M_{GL}(\cdots)}_{\mbox{simple pole}}\\
	&|det(\cdot)|^{\frac{1}{2}}\tau_{a}\times\mbox{$\underbrace{|det(\cdot)|^{\frac{1}{2}}\tau_{a}\rtimes\sigma}$} \ar[d]^{M(s,\tau_a,\sigma)}_{subrep.\mapsto 0}\\ 
	&\mbox{$\underbrace{|det(\cdot)|^{\frac{1}{2}}\tau_{a}\times|det(\cdot)|^{-\frac{1}{2}}\tau_a}$}\rtimes \sigma  \ar[d]^{M_{GL}(\cdots)}_{holomorphic}\\
	|det(\cdot)|^{-\frac{1}{2}}\tau_a\rtimes \sigma_{r} \ar@{^(->}[r] & |det(\cdot)|^{-\frac{1}{2}}\tau_a\times|det(\cdot)|^{\frac{1}{2}}\tau_{a}\rtimes\sigma.
}\]
It is easy to see that the middle arrow associated to $M(s, \tau_a, \sigma)$ maps $|det(\cdot)|^{\frac{1}{2}}\rtimes \sigma_r$ to zero, but the first arrow on the right hand side is a scalar with a simple pole, then the composition of those two arrows is holomorphic, thus $M(s,\tau_a,\sigma_r)$ is holomorphic at $Re(s)=\frac{1}{2}$. Whence our Main Theorem holds by induction. 
\end{proof}	
\begin{rem}\label{simplifyargument}
	Based on the fact that the standard intertwining operator of a standard module is always well-defined and non-zero, and the fact that the normalization factor is non-zero for $Re(s)>0$, we know that the normalized intertwining operator is always well-defined and non-zero for $Re(s)>0$. Thus the discussion on the holomorphy for $Re(s)>0$ could be omitted. Indeed, to prove Casselman--Shahidi's holomorphicity conjecture, we only need to choose one appropriate reduced decomposition such that the corresponding discrepancy only has non-negative poles, and one can readily see that such a decomposition really exists as calculated in the paper. In view of this, our proof could be simplified further. But our argument is originated from attacking the holomorphicity problem of a large class of induced representations involved in the generalized doubling method, and its application to the proof of the Casselman--Shahidi conjecture is just an accidental by-product, so we still want to keep it there to illustrate the idea how we handle analogous problems for general induced representations, for example, replacing discrete series $\tau_a$ by Speh representations $Speh_c(\tau_a)$. 
\end{rem}

\section{main theorem \ref{mainthm} for groups of classical type}\label{sectiongeneral}
In this section, we would like to extend our Main Theorem \ref{mainthm} to the setting of groups of classical type. The main idea is to reduce it to the special case of classical groups which is proved earlier. 

Let $G$ be a quasi-split connected reductive group defined over $F$. Fix a Borel subgroup $B=TU$ of $G$, and denote $A_T$ to be the maximal split torus in the Levi subgroup $T$. Let $Q=LV$ be a standard parabolic subgroup of $G$ with $L$ its Levi subgroup and $V$ its unipotent radical, we denote $\bar{Q}=L\bar{V}$ to be the opposite parabolic subgroup of $Q$. For all the notions below, please refer to \cite{silberger2015introduction,waldspurger2003formule,casselman1995introduction} for the details. 

\subsection{Root datum}Let $X(L)_F$ be the group of $F$-rational characters of $L$, and $A_L$ be the maximal $F$-split sub-torus of the center $Z_L$ of $L$. We set 
\[\mathfrak{a}_L=Hom(X(L)_F,\mathbb{R}),\qquad\mathfrak{a}^\star_{L,\mathbb{C}}=\mathfrak{a}^\star_L\otimes_\mathbb{R} \mathbb{C}, \]
where
\[\mathfrak{a}^\star_L=X(L)_F\otimes_\mathbb{Z}\mathbb{R}=Hom_{alg}(A_L,\mathbb{G}_m)\otimes_\mathbb{Z}\mathbb{R} \]
denotes the dual of $\mathfrak{a}_L$. Recall that the Harish-Chandra homomorphism $H_Q:L\longrightarrow\mathfrak{a}_L$ is defined by
\[q^{\left< \chi,H_Q(m)\right>}=|\chi(m)| \] 
for all $\chi\in X(L)_F$ and $m\in L$. Here $\left<\cdot,\cdot\right>$ is the natural pairing for $\mathfrak{a}_L^*\times \mathfrak{a}_L$. For $\nu \in \mathfrak{a}_{L,\mathbb{C}}^*$, it is viewed as a character of $L$ of the form $q^{\left<\nu,H_L(\cdot)\right>}$. 

Next, let $\Phi$ be the root system of $G$ with respect to $T$, i.e., $A_T$, and $\Delta$ be the set of simple roots determined by $U$. For $\alpha\in \Phi$, we denote by $\alpha^\vee$ the associated coroot, and by $\omega_\alpha$ the associated reflection in the Weyl group $W^G$ of $T$ in $G$ with
\[W^G:=N_G(A_T)/C_G(A_T)=N_G(T)/T=\left<\omega_\alpha:~\alpha\in\Phi\right>. \]
Denote by $w_0^G$ the longest Weyl element in $W^G$, and similarly by $w_0^L$ the longest Weyl element in the Weyl group $W^L$ of a Levi subgroup $L$.

Likewise, we denote by $\Phi_L^G$ the set of reduced relative roots of $L$ in $G$, i.e., $(\Phi|_{A_L})_{red}$, specifically
\[\Phi_L^G:=\{\alpha|_{A_L}: ~\alpha\in \Phi,~ \alpha|_{A_L}\neq c\beta|_{A_L} \mbox{ for some }\beta\in \Phi \mbox{ and some integer }c>1 \}. \]
We denote by $\Delta_L$ the set of relative simple roots determined by $V$ and by $$W_L^G:=N_G(A_L)/C_G(A_L)=N_G(L)/L$$ the relative Weyl group of $L$ in $G$. In general, a relative reflection $\omega_\alpha:=w_0^{L_\alpha}w_0^L$ with respect to a relative root $\alpha$ does not preserve our Levi subgroup $L$, here $L_\alpha\supset L$ is the co-rank one Levi subgroup associated to $\alpha$ w.r.t. $L$ in $G$. In particular, if $L=T$, then this $\omega_{\alpha}$ is the same one defined earlier. For simplicity, we will use the same terminology if no confusion arises.

\subsection{Parabolic induction}For $Q=LV$ a parabolic subgroup of $G$ and an admissible representation $(\sigma,V_\sigma)$ of $L$, we have the following normalized parabolic induction of $Q$ to $G$ which is an admissible representation of $G$
\[Ind_Q^G(\sigma):=\big\{\mbox{smooth }f:G\rightarrow V_\sigma|~f(nmg)=\delta_Q(m)^{1/2}\sigma(m)f(g), \forall n\in V, m\in L\mbox{ and }g\in G\big\}, \]
here $\delta_Q$ stands for the modulus character of $Q$, i.e., denote by $\mathfrak{v}$ the Lie algebra of $V$,
\[\delta_Q(nm)=|det~Ad_\mathfrak{v}(m)|. \]
For a rational character $\nu\in \mathfrak{a}_L^*$ of $L$, we write $Ind^G_Q(\sigma_\nu)$ for the induced representation $Ind^G_Q(\sigma\otimes \nu)$. Define the action of $w\in W_L$ on a representation $\sigma$ of $L$ to be $w.\sigma:=\sigma\circ Ad(w^{-1})$, and $w.Q:=Ad(w).Q=wQw^{-1}$. 

\subsection{Whittaker model}\label{genericnotion}For each root $\alpha\in\Phi$, there exists a non-trivial homomorphism $X_\alpha$ of $F$ into $G$ such that, for $t\in T$ and $x\in F$, 
\[tX_\alpha(x)t^{-1}=X_\alpha(\alpha(t)x). \]
We say a character $\theta$ of $U$ is \underline{generic} if the restriction of $\theta$ to $X_\alpha(F)$ is non-trivial for each simple root $\alpha\in \Delta$. Then the Whittaker function space $\mathcal{W}_\theta$ of $G$ with respect to $\theta$ is the space of smooth complex functions $f$ on $G$ satisfying, for $u\in U$ and $g\in G$,
\[f(ug)=\theta(u)f(g), \]
i.e., $\mathcal{W}_\theta=Ind_U^G(\theta)$. We say an irreducible admissible representation $\pi$ of $G$ is \underline{$\theta$-generic} if
\[\pi\xrightarrow[non-trivial]{G-equiv.} \mathcal{W}_\theta. \]
By \cite[Section 3]{shahidi1990proof}, one can fix a generic character $\theta$ of $U$ such that it is compatible with $w_0^Gw_0^L$ for every Levi subgroup $L$. For simplicity, we also denote by $\theta$ the restriction of $\theta$ to $L\cap U$ if there is no confusion. Every generic representation of $L$ becomes
generic with respect to $\theta$ after changing the splitting in $U$. In view of this, throughout the paper for simplicity, we will only say a representation is generic without specifying its dependence on $\theta$.

\subsection{Casselman--Shahidi conjecture}
Let $P_0=M_0N_0\subset G$ (resp. $P=MN\supset P_0=M_0N_0$) be a standard (resp. maximal) parabolic subgroup with $M_0$ (resp. $M$) its Levi subgroup, $\rho$ be a unitary generic supercuspidal representation of $M_0$, and $\nu\in \mathfrak{a}_{M_0}^*$, we have an induced representation $Ind^M_{M\cap P_0}(\rho_\nu)$. Assume that the induced representation contains a generic discrete series constituent $\sigma$. Let $r$ be the adjoint action of the $L$-group $^{L}M$ of $M$ on the Lie algebra $^{L}\mathfrak{n}$ of the $L$-group of $N$. Then $r=\oplus_{i=1}^{t} r_i$, with $r_i$ irreducible for $i=1,\cdots,t$. Such a decomposition is ordered according to the order of eigenvalues of $^{L}A_M$ in $^{L}N$. Attached to $r_i$ and the generic discrete series representation $\sigma$, we have an L-function $L(s,\sigma,r_i)$ defined by F. Shahidi (see \cite{shahidi1990proof} for the details). 

Next, assume that $M$ is generated by the subset $\Theta=\Delta\backslash\{\alpha_0\}$ of simple roots $\Delta$ of $A_T$ in $U$ for some $\alpha_0\in \Delta$. Recall that the simple reflection $\omega_M=w^G_0w^M_0$, attached to $\alpha_0$, satisfies $\omega_{M}.\Theta\subset \Delta$. Denote $N_{\omega_{M}}:=U\cap \omega_{M}.\bar{N}$ and
$\tilde{\alpha}_0:=\left<\rho_M,\alpha_0^\vee\right>^{-1}\rho_M$, where $\rho_M$ is half the sum of roots in $N$ and $\alpha_0^\vee$ is the co-root of $\alpha_0$. Given $s\in \mathbb{C}$, it is known that $s\tilde{\alpha}_0\in \mathfrak{a}_{M,\mathbb{C}}^*$, and we define the associated standard intertwining operator for $Ind^G_P(\sigma_{s\tilde{\alpha}_0})$ as follows.
\begin{align*}
M(s,\sigma):~&Ind^G_P(\sigma_{s\tilde{\alpha}_0})\longrightarrow Ind^G_{\omega_{M}.\bar{P}}((\omega_{M}.\sigma)_{\omega_{M}.(s\tilde{\alpha}_0)}) \\
&f(g)\mapsto \int_{N_{\omega_{M}}}f(\omega_{M}^{-1}ng)dn.
\end{align*}
It is well known that $M(s,\sigma)$ converges absolutely for $Re(s)> 0$ and extends to a meromorphic function of $s\tilde{\alpha}_0\in\mathfrak{a}_{M,\mathbb{C}}^*$. The knowledge of its poles on all of $\mathfrak{a}_{M,\mathbb{C}}^*$ is very important, and Casselman--Shahidi proposed the following conjecture in \cite[P. 563 (3)]{casselman1998irreducibility} based on the profound Langlands--Shahidi theory. 
\begin{namedthm*}{Casselman--Shahidi's holomorphicity conjecture}
	Keep the notions as before. We have
	\[\prod_{i=1}^{t}L(is,\sigma^\vee,r_i) M(s,\sigma)\]
	is homomorphic in $s\in \mathbb{C}$. Here $\sigma^\vee$ is the contragredient dual of $\sigma$.
\end{namedthm*}
From now on, we fix $G$ to be a quasi-split group of classical type, i.e., Type $A_n$, $B_n$, $C_n$, or $D_n$. The Levi subgroup $M_0$ can be written as two parts $M_0(GL)\times M_0(G)$, here $M_0(G)$ means the part of Type $G$ and $M_0(GL)$ means the part of products of Type $GL$. For two parabolic subgroups $Q=LV\supset Q'=L'V'\supset P_0=M_0N_0$, if formally $L'(G)=M_0(G)$ and $L'(GL)=L(GL)\times GL$, we say that $L'$ is a Siegel Levi subgroup of $L$ relative to $M_0$. Similar notions for root systems. The main result of this section is to show that, via reducing to the special case handled in Main Theorem \ref{mainthm}, please refer to Remark \ref{clarifyabstractargument} for the simple philosophy hidden in the abstract argument which may help you to understand the proof,
\begin{thm}\label{mthgeneral}
	Retain the notation as above. We have 
	\[\mbox{Casselman--Shahidi's holomorphicity conjecture holds for groups of classical type}.\]
\end{thm}
\begin{proof}
As the induced representation $Ind^M_{M\cap P_0}(\rho_\nu)$ contains a discrete series constituent, then by Harish-Chandra's theorem (see \cite[Theorem 3.9.1]{silberger1981discrete}, \cite[Corollary 8.7]{heiermann2004decomposition}, or \cite[Lemma 5.1]{luo2018R}), we know that the sub-root system $\Phi_\rho^M:=\{\alpha\in \Phi^M_{M_0}:~\omega_\alpha.\rho\simeq \rho\}$ should generate the subspace $(\mathfrak{a}^M_{M_0})^*:=Span_\mathbb{R}\{\alpha:~\alpha\in \Phi_{M_0}^M\}$. Moreover, via a simple calculation,
\begin{itemize}
	\item If the root system of $M$ is of Type $A_n$: $\Phi_\rho^M=\Phi^M_{M_0}$.
	\item If the root system of $M$ is of other types: $\Phi_\rho^M$ is the direct sum of root systems of Type $D_n$ or the same type as started. Here $D_n$ can be $D_2$ or $D_3$ if it contains two roots of the form $\{e_i\pm e_j\}$.
\end{itemize}
To be precise, for the case of Type $A_n$, we know first $\Phi_\rho^M$ must be irreducible, otherwise it contracts the dimension equality of the vector space it generated. Secondly, for any simple root $\gamma$, either it is in $\Phi_\rho^M$, or there exist two roots $\alpha$ and $\beta$ in $\Phi_\rho^M$ satisfying that $\gamma=\pm \alpha\pm \beta$ for a choice of signs, thus our claim follows from the fact that $Ad(\omega_\beta).(\omega_\alpha)=\omega_{\omega_\beta.\alpha}=\omega_\gamma$. As for the case of other types, we know that the irreducible pieces, appearing in the decomposition of the root system $\Phi_\rho^M$, are only possible of Type $A_n$, Type $D_n$ or the same type as started, which results from the fact that 
\[Ad(\omega_{e_1-e_2}).\omega_{ce_1}=\omega_{ce_2}\mbox{ and }Ad(\omega_{ce_1}).\omega_{e_1-e_2}=\omega_{e_1+e_2}\qquad (c=1,2). \]
But the dimension equality condition implies that Type $A_n$ can not appear, whence our claim holds.

Up to twisting by a Weyl element, we may assume that $\Delta_{M_0}$ decomposes into canonical blocks w.r.t. $T$, i.e., it is associated to a block partition of $\Delta$, and further assume that $\Phi_\rho^M$ also decomposes into canonical blocks w.r.t. $M_0$, i.e., it is associated to a block partition of $\Delta_{M_0}$. Moreover, by absorbing the unitary part of $s\tilde{\alpha}_0$ into $\rho$, we may only need to consider the holomorphicity problem for $s\in \mathbb{R}$.

On one hand, it is an easy calculation to see that Type $A_n$ can be proved directly as in Lemma \ref{lemGL}. To be precise, consider the reduced decomposition of $\omega_M$ in terms of co-rank one simple reflections with respect to $M_0$, we know that the simple reflection $\omega_{\alpha_0}$ appears only once, thus \[\omega_M.\rho\simeq \rho\mbox{ if and only if }\omega_{\alpha_0}.\rho\simeq \rho.\]
By the well-known fact that the co-rank one Plancherel measure is always a non-zero scalar if $\rho$ is not self-dual, with the help of Shahidi's definition of $L$-functions, it suffices to prove Theorem \ref{mthgeneral} in the setting of $\omega_{M}.\rho\simeq \rho$, which implies that $\Phi_\rho^G=\Phi_{M_0}$. Therefore our problem falls into the case of $GL$ discussed in Lemma \ref{lemGL}. Whence this case is settled.
 
As for other types, write $\Phi_\rho^M=\Phi_\rho^M(GL)\sqcup \Phi_\rho^M(G)=\Phi_\rho^M(GL)\sqcup (\sqcup_j\Phi_\rho^M(G)_j)$, according to the Levi structure of the form $\mbox{Type }A_n\times \mbox{Type }G$ for $M$. Here $\{\Phi_\rho^M(G)_j\}$ are irreducible pieces of $\Phi_\rho^M(G)$. By analyzing the reduced decomposition of $\omega_M$ with respect to the Siegel Levi subgroup of each $\Phi_\rho^M(G)_j$ relative to $\Phi_\rho^M(G)-\Phi_\rho^M(G)_j$, i.e., Way 3: $\omega_M=\omega_{GL}\omega_{M'}\omega_{GL}$, here $\omega_{GL}$ means the $GL$-part intertwining and $\omega_{M'}$ means the same type intertwining as $\omega_{M}$, we can reduce our problem to the irreducible case, i.e., $\Phi_\rho^M(G)$ is irreducible (possibly empty). This follows from the analysis done above for Type $A_n$ and the well-known fact that the $\omega_{GL}$-intertwining operator is always an isomorphism if $\omega_{GL}.\rho\not\simeq \rho$ (see \cite{silberger1980special}). If $\Phi_\rho^M(G)=\emptyset$, then Theorem \ref{mthgeneral} follows from Main Theorem \ref{mainthm} by the fact that the parameters $s$ and $\nu$, and the $L$-functions in the normalization factor can be taken in the same form as in the classical groups case. If $\Phi_\rho^M(G)\neq \emptyset$, consider again the Siegel Levi subgroup of $M$ relative to $M_0$ and the associated reduced decomposition, i.e., Way 3: $\omega_M=\omega_{GL}\omega_{M'}\omega_{GL}$, we know that there are only two cases, either $\omega_{GL}.\rho\simeq \rho$ or $\omega_{GL}.\rho\not\simeq \rho$. The latter case can be reduced to the proved situation $\Phi_\rho^M(G)= \emptyset$, while the former case implies that, by the above analysis done for Type $A_n$, $\rho$ is fixed by all simple reflections in the Siegel Levi subgroup of $G$ relative to $M_0$. Thus it has been reduced to the classical groups case. Whence our Theorem \ref{mthgeneral} holds.
\end{proof}
\begin{rem}\label{clarifyabstractargument}
	To clarify the above abstract argument, we use the classical group $G_n$ as a model to summarize the main points of our argument as follows.
	\begin{itemize}
		\item Levi subgroups $M_0=\left(\prod_{i}GL_{\iota_i}\right)\times\left(\prod_{m}\left(\prod_{j}GL_{m_j}\right)\times G_{n_{00}}\right)\subset M=GL_\iota\times G_{n_0}\subset G_{n}$ with further Levi subgroup relations
		\[\prod_{i}GL_{\iota_i}\subset GL_\iota\mbox{ and }\prod_{m}\left(\prod_{j}GL_{m_j}\right)\times G_{n_{00}}\subset G_{n_0}. \]
		\item $\sigma$ discrete series of $M$ is a constituent of $Ind^M(\rho_\nu)$ with $\rho$ supercuspidal and $\nu$ an unramified character of $M_0$, i.e.,
		\[\sigma=\sigma_\iota\otimes \sigma_{n_0}\mbox{ and }\rho=(\otimes_{i}\rho_{\iota_i})\otimes (\otimes_m(\otimes_j \rho_{m_j})\otimes \rho_{n_{00}}) \]
	\end{itemize}
Thus we have 
\begin{enumerate}[(i).]
	\item Analysis done for Type $A_n$ is equivalent to saying that
	\[\rho_{\iota_i}\simeq \rho_{\iota_{i'}}\mbox{ for any }i,~i'. \]
	\item $\Phi_\rho^M(G)=\sqcup_m\phi_{\rho}^M(G)_m$ is equivalent to saying that
	\[\rho_{m_j}\simeq \rho_{m_{j'}}\mbox{ but }\rho_{m_j}\not \simeq\rho_{m'_{j'}}\mbox{ for any }j,~j', \mbox{ and }m\neq m'.  \] 
	\item The reduction steps are to consider the reduced decomposition given by the embedding  $$\sigma_{n_0}\hookrightarrow (\times_j(\rho_{m'_j}\otimes \nu_{m'_j}))\rtimes (\times_{m\neq m'}\times_j(\rho_{m_j}\otimes\nu_{m_j})\rtimes \rho_{n_{00}}).$$
By doing so, via (i), we know that those $m$ such that $\rho_{m_j}\not\simeq\rho_{\iota_i}$ can be removed, i.e., it reduces to the case that $\Phi_\rho^M(G)_m$ is irreducible with $\rho_{m_j}\simeq\rho_{\iota_i}$ if non-empty.
\end{enumerate}
\end{rem}

\paragraph*{\textbf{Acknowledgments}} The author would like to dedicate this work to Professor Wee Teck Gan on the occasion of his upcoming 50th birthday for his generous support and guidance during the Ph.D study at NUS, Singapore. The author also thanks Eyal Kaplan for his kindness and help at BIU, Israel. Thanks are also due to the referee for his/her detailed comments. The author was supported by the ISRAEL SCIENCE FOUNDATION Grant Number 376/21.
			
\bibliographystyle{amsalpha}
\bibliography{ref}

\providecommand{\bysame}{\leavevmode\hbox to3em{\hrulefill}\thinspace}
\providecommand{\MR}{\relax\ifhmode\unskip\space\fi MR }
\providecommand{\MRhref}[2]{%
  \href{http://www.ams.org/mathscinet-getitem?mr=#1}{#2}
}
\providecommand{\href}[2]{#2}
\begin{thebibliography}{CKPSS04}

\bibitem[AP05]{aubertplymen2005plancherel}
Anne-Marie Aubert and Roger Plymen, \emph{Plancherel measure for {${\rm
  GL}(n,F)$} and {${\rm GL}(m,D)$}: explicit formulas and {B}ernstein
  decomposition}, J. Number Theory \textbf{112} (2005), no.~1, 26--66.
  \MR{2131140}

\bibitem[Art84]{arthur1984problemtraceformula}
James Arthur, \emph{On some problems suggested by the trace formula}, Lie group
  representations, {II} ({C}ollege {P}ark, {M}d., 1982/1983), Lecture Notes in
  Math., vol. 1041, Springer, Berlin, 1984, pp.~1--49. \MR{748504}

\bibitem[Art89]{arthur1989unipotentautomorphicconjecture}
\bysame, \emph{Unipotent automorphic representations: conjectures}, no.
  171-172, 1989, Orbites unipotentes et repr\'{e}sentations, II, pp.~13--71.
  \MR{1021499}

\bibitem[Art90]{arthur1990unipotentmotivation}
\bysame, \emph{Unipotent automorphic representations: global motivation},
  Automorphic forms, {S}himura varieties, and {$L$}-functions, {V}ol. {I}
  ({A}nn {A}rbor, {MI}, 1988), Perspect. Math., vol.~10, Academic Press,
  Boston, MA, 1990, pp.~1--75. \MR{1044818}

\bibitem[Art13]{arthur2013endoscopic}
\bysame, \emph{{The Endoscopic Classification of Representations Orthogonal and
  Symplectic Groups}}, vol.~61, American Mathematical Soc., 2013.

\bibitem[Ato20]{atobe2020construction}
Hiraku Atobe, \emph{Construction of local {A}-packets}, arXiv preprint
  arXiv:2012.07232 (2020).

\bibitem[Bad08]{badulescu2008}
Alexandru~Ioan Badulescu, \emph{Global {J}acquet-{L}anglands correspondence,
  multiplicity one and classification of automorphic representations}, Invent.
  Math. \textbf{172} (2008), no.~2, 383--438, With an appendix by Neven Grbac.
  \MR{2390289}

\bibitem[BK00]{bravermankazhdan2000gammafunctions}
Alexander Braverman and David Kazhdan, \emph{{$\gamma$}-functions of
  representations and lifting}, no. Special Volume, Part I, 2000, With an
  appendix by V. Vologodsky, GAFA 2000 (Tel Aviv, 1999), pp.~237--278.
  \MR{1826255}

\bibitem[BK02]{bravermankazhdan2002intertwining}
\bysame, \emph{Normalized intertwining operators and nilpotent elements in the
  {L}anglands dual group}, vol.~2, 2002, Dedicated to Yuri I. Manin on the
  occasion of his 65th birthday, pp.~533--553. \MR{1988971}

\bibitem[BZ77]{bernstein1977induced}
I.N Bernstein and Andrey~V Zelevinsky, \emph{{Induced representations of
  reductive p-adic groups. I}}, Ann. Sci. {\'E}co. Norm. Sup.(4) \textbf{10}
  (1977), no.~4, 441--472.

\bibitem[Cas95]{casselman1995introduction}
William Casselman, \emph{{Introduction to the theory of admissible
  representations of reductive $p$-adic groups}}, Preprint (1995).

\bibitem[CFGK19]{cai2019doubling}
Yuanqing Cai, Solomon Friedberg, David Ginzburg, and Eyal Kaplan,
  \emph{{Doubling constructions and tensor product $L$-functions: the linear
  case}}, Inv. Math. \textbf{217} (2019), no.~3, 985--1068.

\bibitem[CFK18]{cai2018doubling}
Yuanqing Cai, Solomon Friedberg, and Eyal Kaplan, \emph{{Doubling
  constructions: local and global theory, with an application to global
  functoriality for non-generic cuspidal representations}}, arXiv:1802.02637
  (2018).

\bibitem[CKPSS01]{cogdellkimshapiroshahidi2001}
J.~W. Cogdell, H.~H. Kim, I.~I. Piatetski-Shapiro, and F.~Shahidi, \emph{On
  lifting from classical groups to {${\rm GL}_N$}}, Publ. Math. Inst. Hautes
  \'{E}tudes Sci. (2001), no.~93, 5--30. \MR{1863734}

\bibitem[CKPSS04]{cogdellkimshapiroshahidi2004}
\bysame, \emph{Functoriality for the classical groups}, Publ. Math. Inst.
  Hautes \'{E}tudes Sci. (2004), no.~99, 163--233. \MR{2075885}

\bibitem[CPSS11]{cogdellshapiroshahidi2011}
J.~W. Cogdell, I.~I. Piatetski-Shapiro, and F.~Shahidi, \emph{Functoriality for
  the quasisplit classical groups}, On certain {$L$}-functions, Clay Math.
  Proc., vol.~13, Amer. Math. Soc., Providence, RI, 2011, pp.~117--140.
  \MR{2767514}

\bibitem[CS80]{casselman1980unramified}
William Casselman and Joseph Shalika, \emph{{The unramified principal series of
  $p$-adic groups. II. The Whittaker function}}, Compos. Math. \textbf{41}
  (1980), no.~2, 207--231.

\bibitem[CS98]{casselman1998irreducibility}
William Casselman and Freydon Shahidi, \emph{{On irreducibility of standard
  modules for generic representations}}, Ann. Sci. {\'E}co. Norm. Sup{\'e}r
  \textbf{31} (1998), no.~4, 561--589.

\bibitem[DKV84]{delignekazhdanvigneras1984}
P.~Deligne, D.~Kazhdan, and M.-F. Vign\'{e}ras, \emph{Repr\'{e}sentations des
  alg\`ebres centrales simples {$p$}-adiques}, Representations of reductive
  groups over a local field, Travaux en Cours, Hermann, Paris, 1984,
  pp.~33--117. \MR{771672}

\bibitem[FG99]{friedberggoldberg1999nongeneric}
Solomon Friedberg and David Goldberg, \emph{On local coefficients for
  non-generic representations of some classical groups}, Compositio Math.
  \textbf{116} (1999), no.~2, 133--166. \MR{1686785}

\bibitem[GI14]{ganichino2014}
Wee~Teck Gan and Atsushi Ichino, \emph{Formal degrees and local theta
  correspondence}, Invent. Math. \textbf{195} (2014), no.~3, 509--672.
  \MR{3166215}

\bibitem[GL21]{getzliu2021refined}
Jayce~Robert Getz and Baiying Liu, \emph{A refined poisson summation formula
  for certain braverman-kazhdan spaces}, Science China Mathematics \textbf{64}
  (2021), no.~6, 1127--1156.

\bibitem[Grb09]{grbac2009residual}
Neven Grbac, \emph{Residual spectra of split classical groups and their inner
  forms}, Canad. J. Math. \textbf{61} (2009), no.~4, 779--806. \MR{2541385}

\bibitem[Han18]{hanzer2018nonsiegeleisenstein}
Marcela Hanzer, \emph{Non-{S}iegel {E}isenstein series for symplectic groups},
  Manuscripta Math. \textbf{155} (2018), no.~1-2, 229--302. \MR{3742780}

\bibitem[Hei04]{heiermann2004decomposition}
Volker Heiermann, \emph{{D{\'e}composition spectrale et repr{\'e}sentations
  sp{\'e}ciales d'un groupe r{\'e}ductif $p$-adique}}, J. Inst. Math. Jussieu
  \textbf{3} (2004), no.~3, 327--395.

\bibitem[HKS96]{harriskudlasweet1996}
Michael Harris, Stephen~S. Kudla, and William~J. Sweet, \emph{Theta dichotomy
  for unitary groups}, J. Amer. Math. Soc. \textbf{9} (1996), no.~4, 941--1004.
  \MR{1327161}

\bibitem[HM07]{heiermann2007standard}
Volker Heiermann and Goran Mui{\'c}, \emph{{On the standard modules
  conjecture}}, Math. Z. \textbf{255} (2007), no.~4, 847--853.

\bibitem[HO13]{heiermann2013tempered}
Volker Heiermann and Eric Opdam, \emph{{On the tempered $L$-functions
  conjecture}}, Amer. J. Math. \textbf{135} (2013), no.~3, 777--799.

\bibitem[Ike92]{ikeda1992location}
Tamotsu Ikeda, \emph{{On the location of poles of the triple L-functions}},
  Compositio Math. \textbf{83} (1992), no.~2, 187--237.

\bibitem[Jac84]{jacquet1984residualGL}
Herv\'{e} Jacquet, \emph{On the residual spectrum of {${\rm GL}(n)$}}, Lie
  group representations, {II} ({C}ollege {P}ark, {M}d., 1982/1983), Lecture
  Notes in Math., vol. 1041, Springer, Berlin, 1984, pp.~185--208. \MR{748508}

\bibitem[Jan97]{jantzen1997supports}
Chris Jantzen, \emph{{On supports of induced representations for symplectic and
  odd-orthogonal groups}}, Amer. J. Math. (1997), 1213--1262.

\bibitem[Jan00a]{jantzen2000squareintegrable}
\bysame, \emph{On square-integrable representations of classical {$p$}-adic
  groups}, Canad. J. Math. \textbf{52} (2000), no.~3, 539--581. \MR{1758232}

\bibitem[Jan00b]{jantzen2000squareintegrableII}
\bysame, \emph{On square-integrable representations of classical {$p$}-adic
  groups. {II}}, Represent. Theory \textbf{4} (2000), 127--180. \MR{1789464}

\bibitem[JK01]{jantzenkim2001intertwiningimage}
Chris Jantzen and Henry~H. Kim, \emph{Parametrization of the image of
  normalized intertwining operators}, Pacific J. Math. \textbf{199} (2001),
  no.~2, 367--415. \MR{1847139}

\bibitem[JL14]{jantzenliu2014}
Chris Jantzen and Baiying Liu, \emph{The generic dual of {$p$}-adic split
  {$SO_{2n}$} and local {L}anglands parameters}, Israel J. Math. \textbf{204}
  (2014), no.~1, 199--260. \MR{3273456}

\bibitem[JL20]{jantzenluo}
Chris Jantzen and Caihua Luo, \emph{{On supports of induced representations for
  $p$-adic special orthogonal and general spin groups}}, Preprint (2020).

\bibitem[JLZ13]{jiangliuzhang2013}
Dihua Jiang, Baiying Liu, and Lei Zhang, \emph{Poles of certain residual
  {E}isenstein series of classical groups}, Pacific J. Math. \textbf{264}
  (2013), no.~1, 83--123. \MR{3079762}

\bibitem[JLZar]{jiangluozhang2020Symplectic}
Dihua Jiang, Zhilin Luo, and Lei Zhang, \emph{Harmonic analysis and gamma
  functions on symplectic groups}, Mem. Amer. Math. Soc. (To appear).

\bibitem[JS03]{jiangsoudry2003}
Dihua Jiang and David Soudry, \emph{The local converse theorem for {${\rm
  SO}(2n+1)$} and applications}, Ann. of Math. (2) \textbf{157} (2003), no.~3,
  743--806. \MR{1983781}

\bibitem[JS12]{jiangsoudry2012}
\bysame, \emph{Appendix: {O}n the local descent from {${\rm GL}(n)$} to
  classical groups}, Amer. J. Math. \textbf{134} (2012), no.~3, 767--772.
  \MR{2931223}

\bibitem[Kap13]{kaplan2013gcd}
Eyal Kaplan, \emph{On the gcd of local {R}ankin-{S}elberg integrals for even
  orthogonal groups}, Compos. Math. \textbf{149} (2013), no.~4, 587--636.
  \MR{3049697}

\bibitem[Kim01]{kim2001residualOrthogonal}
Henry~H. Kim, \emph{Residual spectrum of odd orthogonal groups}, Internat.
  Math. Res. Notices (2001), no.~17, 873--906. \MR{1859343}

\bibitem[Kim05]{kim2005intertwiningLfunction}
\bysame, \emph{On local {$L$}-functions and normalized intertwining operators},
  Canad. J. Math. \textbf{57} (2005), no.~3, 535--597. \MR{2134402}

\bibitem[Kim17]{kim2017residualU(2n)}
\bysame, \emph{The residual spectrum of {$U(n,n)$}; contribution from {B}orel
  subgroups}, Bull. Iranian Math. Soc. \textbf{43} (2017), no.~4, 191--219.
  \MR{3711828}

\bibitem[KK04]{kimkrishnamurthy2004}
Henry~H. Kim and Muthukrishnan Krishnamurthy, \emph{Base change lift for odd
  unitary groups}, Functional analysis {VIII}, Various Publ. Ser. (Aarhus),
  vol.~47, Aarhus Univ., Aarhus, 2004, pp.~116--125. \MR{2127169}

\bibitem[KK05]{kimkrishnamurthy2005}
\bysame, \emph{Stable base change lift from unitary groups to {${\rm GL}_n$}},
  IMRP Int. Math. Res. Pap. (2005), no.~1, 1--52. \MR{2149370}

\bibitem[KK11]{kim2011local}
Henry~H Kim and Wook Kim, \emph{{On local L-functions and normalized
  intertwining operators II; quasi-split groups}}, On Certain L-functions, Clay
  Math. Proc \textbf{13} (2011), 265--295.

\bibitem[KR92]{kudlarallis1992}
Stephen~S. Kudla and Stephen Rallis, \emph{Ramified degenerate principal series
  representations for {${\rm Sp}(n)$}}, Israel J. Math. \textbf{78} (1992),
  no.~2-3, 209--256. \MR{1194967}

\bibitem[KS97]{kudlasweet1997}
Stephen~S. Kudla and W.~Jay Sweet, Jr., \emph{Degenerate principal series
  representations for {${\rm U}(n,n)$}}, Israel J. Math. \textbf{98} (1997),
  253--306. \MR{1459856}

\bibitem[Laf14]{lafforgue2014Lfunctions}
Laurent Lafforgue, \emph{Noyaux du transfert automorphe de {L}anglands et
  formules de {P}oisson non lin\'{e}aires}, Jpn. J. Math. \textbf{9} (2014),
  no.~1, 1--68. \MR{3173438}

\bibitem[Li18]{li2018zetaintegrals}
Wen-Wei Li, \emph{Zeta integrals, {S}chwartz spaces and local functional
  equations}, Lecture Notes in Mathematics, vol. 2228, Springer, Cham, 2018.
  \MR{3839636}

\bibitem[Luo20]{luo2020knapp}
Caihua Luo, \emph{{Knapp--Stein dimension theorem for finite central covering
  groups}}, Pacific J. Math. \textbf{306} (2020), no.~1, 265--280.

\bibitem[Luo21a]{luo2021holomorphicityClassicalGroup}
\bysame, \emph{Holomorphicity of normalized intertwining operators for certain
  induced representations {$I$}: A toy example}, Preprint (2021).

\bibitem[Luo21b]{luo2021reducibilitypointsGL}
\bysame, \emph{Location of reducibility points of induced representations
  {$I$}: A toy example}, Preprint (2021).

\bibitem[Luo21c]{luo2021gcd}
\bysame, \emph{{On the G.C.D of local generalized doubling integrals}},
  Preprint (2021).

\bibitem[Luo21d]{luo2018R}
\bysame, \emph{{Rodier type theorem for generalized principal series}}, Math.
  Z. \textbf{299} (2021), 897--918.

\bibitem[Mil13]{miller2013residual}
Stephen~D. Miller, \emph{Residual automorphic forms and spherical unitary
  representations of exceptional groups}, Ann. of Math. (2) \textbf{177}
  (2013), no.~3, 1169--1179. \MR{3034297}

\bibitem[M{\oe}g91a]{moeglin1991discretespectrumUnipotent}
Colette M{\oe}glin, \emph{Orbites unipotentes et spectre discret non
  ramifi\'{e}: le cas des groupes classiques d\'{e}ploy\'{e}s}, Compositio
  Math. \textbf{77} (1991), no.~1, 1--54. \MR{1091891}

\bibitem[M{\oe}g91b]{moeglin1991ICM}
\bysame, \emph{Sur les formes automorphes de carr\'{e} int\'{e}grable},
  Proceedings of the {I}nternational {C}ongress of {M}athematicians, {V}ol.
  {I}, {II} ({K}yoto, 1990), Math. Soc. Japan, Tokyo, 1991, pp.~815--819.
  \MR{1159268}

\bibitem[M{\oe}g01]{moeglin2001residualconjecture}
\bysame, \emph{Conjectures sur le spectre r\'{e}siduel}, J. Math. Soc. Japan
  \textbf{53} (2001), no.~2, 395--427. \MR{1815141}

\bibitem[M{\oe}g02]{moeglin2002classification}
\bysame, \emph{Sur la classification des s\'{e}ries discr\`etes des groupes
  classiques {$p$}-adiques: param\`etres de {L}anglands et exhaustivit\'{e}},
  J. Eur. Math. Soc. (JEMS) \textbf{4} (2002), no.~2, 143--200. \MR{1913095}

\bibitem[M{\oe}g06a]{moeglin2006paquets}
\bysame, \emph{Paquets d'{A}rthur pour les groupes classiques; point de vue
  combinatoire}, arXiv preprint math/0610189 (2006).

\bibitem[M{\oe}g06b]{moeglin2006certainpackets}
\bysame, \emph{Sur certains paquets d'{A}rthur et involution
  d'{A}ubert-{S}chneider-{S}tuhler g\'{e}n\'{e}ralis\'{e}e}, Represent. Theory
  \textbf{10} (2006), 86--129. \MR{2209850}

\bibitem[M{\oe}g08]{moeglin2008automorphicsquareintegrable}
\bysame, \emph{Formes automorphes de carr\'{e} int\'{e}grable non cuspidales},
  Manuscripta Math. \textbf{127} (2008), no.~4, 411--467. \MR{2457189}

\bibitem[M{\oe}g09]{moeglin2009arthurpackets}
\bysame, \emph{Paquets d'{A}rthur discrets pour un groupe classique
  {$p$}-adique}, Automorphic forms and {$L$}-functions {II}. {L}ocal aspects,
  Contemp. Math., vol. 489, Amer. Math. Soc., Providence, RI, 2009,
  pp.~179--257. \MR{2533005}

\bibitem[M{\oe}g10]{moeglin2010holomorphy}
\bysame, \emph{Holomorphie des op\'{e}rateurs d'entrelacement normalis\'{e}s
  \`a l'aide des param\`etres d'{A}rthur}, Canad. J. Math. \textbf{62} (2010),
  no.~6, 1340--1386. \MR{2760663}

\bibitem[M{\oe}g11a]{moeglin2011intertwiningeisenstein}
\bysame, \emph{Image des op\'{e}rateurs d'entrelacements normalis\'{e}s et
  p\^{o}les des s\'{e}ries d'{E}isenstein}, Adv. Math. \textbf{228} (2011),
  no.~2, 1068--1134. \MR{2822218}

\bibitem[M{\oe}g11b]{moeglin2011multiplicityone}
\bysame, \emph{Multiplicit\'{e} 1 dans les paquets d'{A}rthur aux places
  {$p$}-adiques}, On certain {$L$}-functions, Clay Math. Proc., vol.~13, Amer.
  Math. Soc., Providence, RI, 2011, pp.~333--374. \MR{2767522}

\bibitem[Mok15]{mok2015endoscopic}
Chung~Pang Mok, \emph{Endoscopic classification of representations of
  quasi-split unitary groups}, Mem. Amer. Math. Soc. \textbf{235} (2015),
  no.~1108, vi+248. \MR{3338302}

\bibitem[MR18a]{moeglinrenard2018arthurpacketsReal}
Colette M{\oe}glin and David Renard, \emph{Sur les paquets d'{A}rthur aux
  places r\'{e}elles, translation}, Geometric aspects of the trace formula,
  Simons Symp., Springer, Cham, 2018, pp.~299--320. \MR{3969879}

\bibitem[MR18b]{moeglinrenard2018arthurpackets}
\bysame, \emph{Sur les paquets d'{A}rthur des groupes classiques et unitaires
  non quasi-d\'{e}ploy\'{e}s}, Relative aspects in representation theory,
  {L}anglands functoriality and automorphic forms, Lecture Notes in Math., vol.
  2221, Springer, Cham, 2018, pp.~341--361. \MR{3839702}

\bibitem[MR20]{moeglinrenard2020arthurpacketsReal}
\bysame, \emph{Sur les paquets d'{A}rthur des groupes classiques r\'{e}els}, J.
  Eur. Math. Soc. (JEMS) \textbf{22} (2020), no.~6, 1827--1892. \MR{4092900}

\bibitem[MS98]{muicshahidi1998standardmoduleconj}
Goran Mui{\'{c}} and Freydoon Shahidi, \emph{Irreducibility of standard
  representations for {I}wahori-spherical representations}, Math. Ann.
  \textbf{312} (1998), no.~1, 151--165. \MR{1645956}

\bibitem[MS00]{muicsavin2000quaternion}
Goran Mui\'{c} and Gordan Savin, \emph{Complementary series for {H}ermitian
  quaternionic groups}, Canad. Math. Bull. \textbf{43} (2000), no.~1, 90--99.
  \MR{1749954}

\bibitem[MT02]{moeglin2002construction}
Colette M{\oe}glin and Marko Tadi{\'c}, \emph{{Construction of discrete series
  for classical $p$-adic groups}}, J. Amer. Math. Soc. \textbf{15} (2002),
  no.~3, 715--786.

\bibitem[Mui98]{muic1998}
Goran Mui\'{c}, \emph{Some results on square integrable representations;
  irreducibility of standard representations}, Internat. Math. Res. Notices
  (1998), no.~14, 705--726. \MR{1637097}

\bibitem[MW89]{moeglin1989residue}
Colette M{\oe}glin and Jean-Loup Waldspurger, \emph{Le spectre r\'{e}siduel de
  {${\rm GL}(n)$}}, Ann. Sci. \'{E}cole Norm. Sup. (4) \textbf{22} (1989),
  no.~4, 605--674. \MR{1026752}

\bibitem[Ng{\^{o}}14]{ngo2014Lfunction}
Bao~Ch{\^{a}}u Ng{\^{o}}, \emph{On a certain sum of automorphic
  {$L$}-functions}, Automorphic forms and related geometry: assessing the
  legacy of {I}. {I}. {P}iatetski-{S}hapiro, Contemp. Math., vol. 614, Amer.
  Math. Soc., Providence, RI, 2014, pp.~337--343. \MR{3220933}

\bibitem[PSR87]{piatetskishapirotriple}
I.~Piatetski-Shapiro and Stephen Rallis, \emph{Rankin triple {$L$} functions},
  Compositio Math. \textbf{64} (1987), no.~1, 31--115. \MR{911357}

\bibitem[Sak12]{sakellaridis}
Yiannis Sakellaridis, \emph{Spherical varieties and integral representations of
  {$L$}-functions}, Algebra Number Theory \textbf{6} (2012), no.~4, 611--667.
  \MR{2966713}

\bibitem[Sha81]{shahidi1981certain}
Freydoon Shahidi, \emph{{On Certain L-Functions}}, Amer. J. Math. \textbf{103}
  (1981), no.~2, 297--355.

\bibitem[Sha90]{shahidi1990proof}
\bysame, \emph{{A proof of Langlands' conjecture on Plancherel measures:
  complementary series of p-adic groups}}, Ann. of Math. \textbf{132} (1990),
  no.~2, 273--330.

\bibitem[Sha18]{shahidi2018Lfunctions}
\bysame, \emph{On generalized {F}ourier transforms for standard
  {$L$}-functions}, Geometric aspects of the trace formula, Simons Symp.,
  Springer, Cham, 2018, pp.~351--404. \MR{3969881}

\bibitem[Sil79]{silberger2015introduction}
Allan~J Silberger, \emph{{Introduction to Harmonic Analysis on Reductive P-adic
  Groups.(MN-23): Based on Lectures by Harish-Chandra at The Institute for
  Advanced Study, 1971-73}}, Princeton university press, 1979.

\bibitem[Sil80]{silberger1980special}
\bysame, \emph{{Special representations of reductive p-adic groups are not
  integrable}}, Ann. Math. \textbf{111} (1980), no.~3, 571--587.

\bibitem[Sil81]{silberger1981discrete}
\bysame, \emph{{Discrete series and classification for p-adic groups I}}, Amer.
  J. Math. \textbf{103} (1981), no.~6, 1241--1321.

\bibitem[ST15]{soudrytanay2015}
David Soudry and Yaacov Tanay, \emph{On local descent for unitary groups}, J.
  Number Theory \textbf{146} (2015), 557--626. \MR{3267124}

\bibitem[Swe95]{sweet1995}
W.~Jay Sweet, Jr., \emph{A computation of the gamma matrix of a family of
  {$p$}-adic zeta integrals}, J. Number Theory \textbf{55} (1995), no.~2,
  222--260. \MR{1366572}

\bibitem[Tad96]{tadic1996generic}
Marko Tadi\'{c}, \emph{On square integrable representations of classical
  $p$-adic groups}, \url{http://www.hazu.hr/~tadic/b-square-int-96.pdf}, 1996.

\bibitem[Tad98]{tadic1998regular}
Marko Tadi{\'c}, \emph{{On regular square integrable representations of p-adic
  groups}}, Amer. J. Math. \textbf{120} (1998), no.~1, 159--210.

\bibitem[Wal03]{waldspurger2003formule}
Jean-Loup Waldspurger, \emph{{La formule de Plancherel pour les groupes
  p-adiques. D'apres Harish-Chandra}}, J. Inst. Math. Jussieu \textbf{2}
  (2003), no.~02, 235--333.

\bibitem[Xu17]{xu2017moeglinparametrization}
Bin Xu, \emph{On {M}{\oe}glin's parametrization of {A}rthur packets for
  {$p$}-adic quasisplit {$Sp(N)$} and {$SO(N)$}}, Canad. J. Math. \textbf{69}
  (2017), no.~4, 890--960. \MR{3679701}

\bibitem[Xu21]{xu2021combinatorialArthurpackets}
\bysame, \emph{A combinatorial solution to {M}{\oe}glin's parametrization of
  {A}rthur packets for {$p$}-adic quasisplit {$Sp(N)$} and {$O(N)$}}, J. Inst.
  Math. Jussieu \textbf{20} (2021), no.~4, 1091--1204. \MR{4293795}

\bibitem[Yam11]{yamana2011}
Shunsuke Yamana, \emph{Degenerate principal series representations for
  quaternionic unitary groups}, Israel J. Math. \textbf{185} (2011), 77--124.
  \MR{2837129}

\bibitem[Yam14]{yamanatheta}
\bysame, \emph{L-functions and theta correspondence for classical groups},
  Invent. Math. \textbf{196} (2014), no.~3, 651--732. \MR{3211043}

\bibitem[Zel80]{zelevinsky1980induced}
Andrei~V Zelevinsky, \emph{{Induced representations of reductive p-adic groups
  II. On irreducible representations of GL(n)}}, Ann. Sci. {\'E}c. Norm.
  Sup{\'e}r. (4) \textbf{13} (1980), no.~2, 165--210.

\bibitem[Zha97]{zhang1997localintertwining}
Yuanli Zhang, \emph{The holomorphy and nonvanishing of normalized local
  intertwining operators}, Pacific J. Math. \textbf{180} (1997), no.~2,
  385--398. \MR{1487571}

\end{thebibliography}

			
\end{document}